\documentclass{m2an}
\usepackage{csquotes,graphicx,comment,xcolor,bm}
\usepackage{epstopdf}
\usepackage[mathscr]{eucal}
\usepackage{subcaption}
\usepackage[font=footnotesize]{caption}
\usepackage{amsfonts,amsmath,amssymb,mathtools,bm,nicefrac} 
\usepackage{enumitem}
\usepackage{mathalfa,colonequals}
\usepackage{cleveref}
\usepackage{pgfplotstable}
\usepackage{bm}
\usepackage{empheq}
\usepackage{dsfont}
\usepackage{graphicx}
\usepackage{color}
\usepackage{circuitikz}
\usepackage{tikz}
\usetikzlibrary{calc,positioning,shapes}
\usetikzlibrary{patterns,decorations.pathmorphing,decorations.markings}
\usepackage{pgfplots}
\pgfplotsset{compat=newest}
\usepackage[margin=10pt,font=small,labelfont=bf,labelsep=endash]{caption}
\usetikzlibrary{spy}
\usetikzlibrary{intersections, backgrounds}
\usetikzlibrary{circuits}
\usetikzlibrary{circuits.ee.IEC}
\pgfplotsset{compat=newest}
\usepackage{booktabs}

\definecolor{color0}{rgb}{0.12156862745098,0.466666666666667,0.705882352941177}
\definecolor{color1}{rgb}{1,0.498039215686275,0.0549019607843137}
\definecolor{color2}{rgb}{0.172549019607843,0.627450980392157,0.172549019607843}
\definecolor{color3}{rgb}{0.83921568627451,0.152941176470588,0.156862745098039}
\definecolor{color4}{rgb}{0.580392156862745,0.403921568627451,0.741176470588235}
\definecolor{color5}{rgb}{0,0,0}

\definecolor{mycolor1}{rgb}{0.00000,0.44700,0.74100}

\definecolor{mycolor2}{rgb}{0.85000,0.32500,0.09800}
\definecolor{mycolor2light1}{rgb}{0.75000,0.27500,0.15000}    
\definecolor{mycolor2light2}{rgb}{0.85000,0.20000,0.20000}   

\definecolor{mycolor3}{rgb}{0.92900,0.69400,0.12500}
\definecolor{mycolor3light1}{rgb}{0.80000,0.60000,0.20000}   
\definecolor{mycolor3light2}{rgb}{0.95000,0.60000,0.05000}   

\definecolor{mycolor4}{rgb}{0.46600,0.67400,0.18800}
\definecolor{mycolor4light1}{rgb}{0.38000,0.58000,0.20000}
\definecolor{mycolor4light2}{rgb}{0.50000,0.72000,0.30000}

\definecolor{mycolor5}{rgb}{0.49400,0.18400,0.55600}

\crefname{proposition}{\textup{Proposition}}{\textup{Propositions}}
\crefname{assumption}{\textup{Assumption}}{\textup{Assumptions}}
\crefname{lemma}{\textup{Lemma}}{\textup{Lemmas}}
\crefname{algorithm}{\textup{Algorithm}}{\textup{Algorithms}}
\crefname{theorem}{\textup{Theorem}}{\textup{Theorems}}
\crefname{remark}{\textup{Remark}}{\textup{Remarks}}
\crefname{example}{\textup{Example}}{\textup{Examples}}
\crefname{corollary}{\textup{Corollary}}{\textup{Corollaries}}
\crefname{subsection}{\textup{Section}}{\textup{Subsections}}
\crefname{section}{\textup{Section}}{\textup{Sections}}

\usepackage{accents}

\usepackage{graphicx,color,xcolor}
\usepackage{algorithm} 
\usepackage{algorithmic} 

\usepackage{cite}

\makeatletter
\let\@citexOld\@citex
\def\@citex[#1]#2{\textup{\@citexOld[#1]{#2}}}
\makeatother

\usepackage{placeins}
\usepackage{csquotes}

\newcommand{\argmin}{\arg\!\min}
\newcommand{\bx}{{\bm x}}

\newcommand{\N}{\mathbb N}
\makeatletter
\def\namedlabel#1#2{\begingroup
    #2%
    \def\@currentlabel{#2}%
    \phantomsection\label{#1}\endgroup
}

\def\R{\mathbb{R}}

\newcommand{\ddt}{\partial_t}
\newcommand{\bu}{\mathbf{u}}
\newcommand{\bv}{\mathbf{v}}
\newcommand{\dt}{\mathrm{d}t}
\newcommand{\ds}{\ \mathrm{d}s}
\newcommand{\reduce}[1]{{{#1}^r}}

\newcommand{\FOM}{FOM}
\newcommand{\ubar}[1]{\underaccent{\bar}{#1}}
\DeclareMathOperator*{\essinf}{ess\,inf}
\newcommand{\tint}{t_{\mathrm{in}}}
\newcommand{\yint}{y_{\mathrm{in}}}
\newcommand{\tildeyint}{\tilde y_{\mathrm{in}}}

\newcommand{\mU}{{\mathscr U}}
\newcommand{\mX}{{\mathscr X}}
\newcommand{\mY}{{\mathscr Y}}

\newcommand{\calL}{\mathcal{L}}

\newcommand{\calR}{\mathcal{R}}

\newcommand{\Res}{\calR}

\newcommand{\linspan}{\mathop{\rm span}\nolimits}

\newcommand{\rest}{\left.\kern-2\nulldelimiterspace\right|_}
\newcommand{\norm}[2]{\left|#1\right|_{#2}}
\newcommand{\scalprod}[3]{\langle#1, #2\rangle_{#3}}

\definecolor{DarkBlue}{rgb}{0,0.08,0.45}
\definecolor{DarkRed}{rgb}{.65,0,0}
\definecolor{applegreen}{rgb}{0.55, 0.71, 0.0}

\newcounter{mymac@matlab}
\newcommand{\matlab}{MATLAB%
   \ifnum\value{mymac@matlab}<1%
   \textregistered%
   \setcounter{mymac@matlab}{1}%
   \fi%
  }

\providecommand{\argmin}{\operatorname*{argmin}}

\newtheorem{lemma}[thrm]{Lemma}
\newtheorem{assumption}[thrm]{Assumption}
\newtheorem{definition}[thrm]{Definition}
\newtheorem{remark}[thrm]{Remark}
\newtheorem{proposition}[thrm]{Proposition}
\newtheorem{theorem}[thrm]{Theorem}
\newtheorem{corollary}[thrm]{Corollary}
\begin{document}
\title{Stabilization of parabolic time-varying PDEs using certified reduced-order receding horizon control}
\author{Behzad Azmi}\address{Department of Mathematics and Statistics, University of Konstanz, D-78457 Konstanz, \{behzad.azmi, michael.kartmann, stefan.volkwein\}@uni-konstanz.de.}
\author{Michael Kartmann$^1$}
\author{Stefan Volkwein$^1$}
\date{\today}
\begin{abstract} 
We address the stabilization of linear, time-varying parabolic PDEs using finite-dimen\-sional receding horizon controls (RHCs) derived from reduced-order models (ROMs). We first prove exponential stability and suboptimality of the continuous-time full-order model (FOM) RHC scheme in Hilbert spaces. A Galerkin model reduction is then introduced, along with a rigorous a posteriori error analysis for the associated finite-horizon optimal control problems. This results in a ROM-based RHC algorithm that adaptively constructs reduced-order controls, ensuring exponential stability of the FOM closed-loop state and providing computable performance bounds with respect to the infinite-horizon FOM control problem. Numerical experiments with a non-smooth cost functional involving the squared $\ell^1$-norm confirm the method’s effectiveness, even for exponentially unstable systems.
\end{abstract}
%
%
\maketitle

\section{Introduction}
Receding Horizon Control (RHC), also known as Model Predictive Control (MPC), is an optimization-based strategy for solving infinite-horizon optimal control problems. In this framework, the infinite-horizon problem is approximated by a sequence of finite-horizon problems defined over temporally overlapping intervals that cover the time domain \((0,\infty)\). Due to this structure, the resulting control acts as a feedback mechanism, offering an efficient strategy for addressing infinite-horizon problems governed by discrete-time \cite{grune2017nonlinear,GR08} and continuous-time systems \cite{RMD19,RA12}. 

Despite the flexibility of this open-loop optimization approach, establishing theoretical guarantees for stability and suboptimality remains a significant challenge. These issues are typically addressed by incorporating terminal costs and/or constraints, or by carefully designing the overlapping intervals. This framework has also been investigated for problems governed by partial differential equations (PDEs), see e.g., \cite{azmi2019hybrid,IK02}.

From a computational perspective, repeatedly solving PDE-constrained open-loop problems can be extremely costly, primarily due to the large state-space dimension resulting from PDE discretizations. This makes the standard RHC framework computationally expensive or even infeasible in practice. It is therefore essential to accelerate open-loop computations and improve their efficiency while preserving stability and suboptimality guarantees of RHC.

In this paper, we address the stabilization of linear time-varying parabolic PDEs with finite-dimensional controls by combining model reduction techniques and a rigorous analysis of the associated open-loop problems. First, we address the exponential stability and suboptimality of a continuous-time RHC framework in Hilbert spaces. In this setting, no terminal costs or constraints are required; instead, stability is ensured by selecting appropriate concatenation schemes. Furthermore, this framework allows the use of the squared \(\ell^1\)-norm as the control cost, leading to a non-smooth infinite-horizon problem that enforces sparsity in the control input. In the second part, we focus on improving computational efficiency using projection-based Galerkin model order reduction (MOR) techniques \cite{hesthaven2016certified,HKMU23}. Model order reduction (MOR) methods aim to accelerate computations by replacing the high-dimensional full-order model (FOM) with a low-dimensional reduced-order model (ROM). Galerkin ROMs, based on, e.g., Proper Orthogonal Decomposition (POD) \cite{KV02,GubV17,MU22}, are particularly effective for parabolic PDEs, as the Kolmogorov $n$-width can be expected to decay exponentially in this setting~\cite{cohen2011analytic,HKMU23}. However, simply applying a reduced feedback control (i.e., a control computed from the reduced model) to the FOM may compromise closed-loop stability. To address this issue, appropriate conditions on the ROM must be incorporated into the algorithm to ensure the stability of the FOM closed-loop system. In this paper, we build on a rigorous \emph{a posteriori} error estimator for the finite-horizon value function. Together with a Relaxed Dynamic Programming Principle (RDP) (see \cref{theorem:RelaxedDDP}), this allows us to show that the ROM-based RHC guarantees not only exponential stabilization of the FOM but also suboptimality with respect to the original infinite-horizon control problem.
\subsection{Related Works}
Due to its flexibility in handling constraints, non-autonomous dynamics, and non-linearities, RHC has received increased attention for the stabilization of PDE systems; see e.g., ~\cite{IK02,G09,BF20,azmi2019hybrid}. In this paper, we consider an unconstrained RHC framework in Hilbert spaces for stabilizing a general class of continuous-time, linear, time-varying parabolic equations, where the control enters as a linear combination of finitely many indicator functions. As already mentioned, this framework does not require any terminal cost or terminal constraints to guarantee stability. Similar approaches have been studied in the context of continuous-time ODEs \cite{RA12,JH05} and discrete-time dynamical systems~\cite{GR08,G09}.

Many works study the incorporation of RHC (MPC) in combination with MOR, see e.g., \cite{AV15,ARV25,GU14,kartmann2024certifiedmodelpredictivecontrol} and the references therein. Considerable effort has been devoted to establishing conditions under which the stability of FOM is preserved when the control is computed from ROMs. For finite-dimensional systems, we refer to~\cite{loehning2014model,LMFP22,APKCSZHP23}, where the error dynamics of the reduced system are explicitly incorporated into the MPC subproblem, and stability is ensured through suitable terminal conditions.
An alternative approach for discrete-time, unconstrained MPC applied to parameterized autonomous linear parabolic PDEs was proposed in \cite{dietze2023reduced}. In this contribution, a projection-based ROM is constructed offline using a greedy parameter selection strategy. The resulting ROM is then employed online to determine minimal stabilizing prediction horizons. Relying on \emph{a posteriori} error estimates for the value function, the FOM performance index can be estimated efficiently online. If the performance estimator remains positive throughout the RHC iterations, stability follows from the discrete-time RDP~\cite{grune2017nonlinear}. Our performance certification approach is inspired by~\cite{dietze2023reduced}, although we do not consider parameterized discrete-time systems but instead focus on an online-adaptive construction of the ROM for a continuous-time control problem with fixed parameters.

Owing to their structural properties, Galerkin ROMs provide a natural foundation for \emph{a posteriori} error estimation. This has been demonstrated in~\cite{dietze2023reduced,karcher2018certified,ali2020reduced} for optimality systems arising from parameterized control-constrained linear-quadratic elliptic (or discrete-time parabolic) optimal control problems, and in~\cite{ran,vexler2008adaptive} for value function estimation within an adaptive finite element framework. Moreover, efficient MPC implementations for parabolic PDEs with adaptive grid refinement in space or time have also been studied; see, e.g.,~\cite{GSS22,AGH22}.
\subsection{Contributions}
Based on the above discussion, our main contributions can be summarized as follows:
\begin{enumerate}
%

    \item {We establish exponential stability and suboptimality of continuous-time RHC for time-varying linear parabolic systems with a weaker stage cost, namely an $L^2(\Omega)$-tracking term; see \cref{theorem:stabilityRHC}. By contrast, the exponential stability analysis in \cite[Prop.~2.1 and Thm.~6.2]{azmi2019hybrid} requires an $H^1(\Omega)$-tracking term in order to invoke an \emph{initial-time} observability estimate (see \cite[Assumption~\textbf{P3} in Section~2]{azmi2019hybrid}). For an $L^2(\Omega)$-tracking term, \cite[Thm.~6.4]{azmi2019hybrid} establishes only asymptotic stability. Our analysis does not rely on an initial-time observability condition. Instead, we employ a \emph{final-time} observability inequality, namely \cref{e8a}. As a consequence, and in comparison with \cite{azmi2019hybrid,RA12}, we obtain a new explicit bound for the suboptimality (performance) parameter.}

    \item Extending the results of \cite{ran,dietze2023reduced}, we derive rigorous \textit{a posteriori} error estimates for the ROM approximation of the value function and the cost functional, applicable to general non-smooth convex control costs (including sparsity-promoting regularization and convex control constraints), and to inexact initial values {in} \cref{lem:errorest_optimalval_diff_initguess_prepare}.
    Furthermore, we provide error and residual equivalences, establish interpolation properties of the reduced optimality system, and asymptotic convergence of Galerkin ROMs under relaxed regularity assumptions in \cref{theo:os_conv}.
    
    \item We propose a ROM-RHC scheme that combines the \emph{a posteriori} error estimates with a new variant of the RDP principle for time-varying continuous-time systems (see \cref{theorem:RelaxedDDP}), obtaining exponentially stabilizing controls with user-specified minimal performance guarantees relative to the infinite-horizon FOM value function (see \cref{theorem:StabilityReducedRHC}). 
    In particular, we show that one can always construct the ROM such that the RDP inequality is satisfied with a finite-dimensional reduced basis, thereby ensuring suboptimality with a performance parameter strictly smaller than that of the FOM. The results on error estimation, stability, and suboptimality are independent of the specific Galerkin method used for the numerical realization and apply both in MOR and adaptive finite element frameworks.
    
    \item We provide numerical experiments on systems with exponentially unstable free dynamics, demonstrating that ROM-RHC with squared $\ell_1$-regularization achieves substantial speed-ups over FOM-RHC while preserving the stability and suboptimality guarantees of \cref{theorem:StabilityReducedRHC}.     
 \end{enumerate}
\subsection{Outline}
The remainder of this paper is organized as follows. In \cref{sec:FOM_stability_RHC}, we present the RHC scheme and establish its suboptimality and exponential stability for the FOM. Section~\ref{sec:ROM} focuses on the ROM, where we derive \emph{a posteriori} error estimates, present convergence results, and analyze their relation to the true error. Building on these results, \cref{sec:stabROMRHC} develops the relaxed stability framework, yielding a certified ROM-RHC scheme with guaranteed stability. An illustrative example demonstrating the applicability of the framework is given in \cref{sec:example}, and numerical experiments validating the approach are presented in \cref{sec:NUMEXP}.
\subsection{Notation and preliminaries}
We denote by $\R_{>0}$ ($\R_{\geq 0}$) the set of positive (non-negative) real numbers.
Throughout the paper, let $V \hookrightarrow  H=H' \hookrightarrow  V'$ be a Gelfand triple of separable Hilbert spaces $H$ and $V$ with $V$ compactly and densely embedded in $H$. The space $\mathcal L(V,V')$ represents the Banach space of linear and bounded operators from $V$ to $V'$. For $\tint \in \R_{\ge0}$ and $T\in \R_{> 0}\cup\{\infty\}$, we define the control and state spaces as
\begin{align*}
    \mU_T(\tint)&\coloneqq L^2(\tint,\tint+T;U)\quad\text{for }U=\R^m\text{ and }m\in\N,\\
    \mY_T(\tint)&\coloneqq W(\tint,\tint+T;V)\coloneqq\big\{\varphi\in L^2(\tint,\tint+T;V)\,\big|\,\partial_t\varphi\in L^2(\tint,\tint+T;V')\big\}
\end{align*}
with induced norm $\norm{y}{\mY_T(\tint)}^2= \norm{y}{L^2(\tint,\tint+T;V)}^2+\norm{\partial_t y}{{L^2(\tint,\tint+T;V')}}^2$. Recall that $\mY_T(\tint)\hookrightarrow C([\tint,\tint+T];H)$ for $T<\infty$ holds. If it is clear from the context, we abbreviate $L^p(\tint,\tint+T;C)$ by $L^p(C)$ or simply $L^p$ for a Hilbert space $C$ and $p\in [1,\infty]$. Dependence on data is indicated after a semicolon. For instance, we write $y(t;\bu,\tint,\yint)$ for a state variable $y(t)$ depending on time $t$, control $\bu$, and initial value $\yint$ at initial time $\tint$. If it is clear from the context we write, e.g., only $y(\bu)$ to abbreviate $y(\bu,\tint,\yint)$. The subscript (or superscript) $ r$ indicates reduced quantities. For example, we denote a reduced subspace by $V_r\subset V$, and $\reduce{y}\in\mY_T^r(\tint)\coloneqq W(\tint,\tint+T;V_r)$ denotes a reduced state in the reduced state space.
%
\section{The full-order model and suboptimality of RHC}\label{sec:FOM_stability_RHC}
%
For the triple $(T,\tint,\yint)\in {({\R_{>0}\cup \{ \infty \})}}\times\R_{\ge0} \times H$ and a control $\bu\in \mU_T(\tint)$, consider the linear time-varying system
\begin{align}
    \label{eq:LTV}
    \tag{$\mathrm{FOM}_T(\tint,\yint)$}
    \ddt y(t)+A(t)y(t)=B(t)\bu(t), \quad t \in  ( \tint,\tint+T),\quad y(\tint)=\yint.
\end{align}
We refer to \eqref{eq:LTV} as the \emph{full-order model} (FOM). {Note that in the MOR literature, the term FOM is used to indicate a discrete (e.g., FE) model of high dimension. For notational convenience, we do not distinguish between $V$ being a continuous infinite-dimensional space used for theoretical considerations or a high-dimensional FE space used for numerical computations in \cref{sec:NUMEXP}.}
\begin{assumption}\label{ass:pde}
    We assume that $A\in L^\infty(0,\infty;\mathcal{L}(V,V'))$, $B \in  L^\infty(0,\infty;\mathcal{L}(U,V'))$
     and the existence of constants $\eta_V>0$, $\eta_H\geq 0$ such that
    \begin{align}\label{eq:weak_coercivity}
        \scalprod{A(t)v}{v}{V',V} \geq \eta_V\norm{v}{V}^2 - \eta_H\norm{v}{H}^2 \quad \text{for all } v\in V,\  t\geq 0.
    \end{align}
\end{assumption}
Then, for all $T\in\R_{>0}$ and $\bu\in \mU_T(\tint)$, there exists a unique solution $y=y(\bu,\tint,\yint)\in\mY_T(\tint)$; cf. \cite{LioM72}. To specify the optimal control problems, we introduce the cost functional
\begin{equation}\label{eq:cost}\nonumber
    J_T(\bu;\tint,\yint)\coloneqq\int^{\tint+T}_{\tint} \ell(y(t;\bu,\tint,\yint),\bu(t))\ \dt.
\end{equation}
\begin{assumption}\label{ass:running_cost}
The incremental function $\ell: H\times U \to \R_{\geq 0}$ is given as
\begin{equation}
\label{eq:ell_property}
\ell(\varphi,\bv)\coloneqq\tfrac{1}{2}\norm{ \varphi}{H}^2+\tfrac{\lambda}{2}\norm{\bv}{U}^2+g(\bv)\quad\text{for }\,(\varphi,\bv)\in H \times U\text{ and }\lambda>0.
\end{equation}
Hereby, $g:U \to \R_{\geq 0}\cup \{\infty\}$ is proper, convex, and lower-semicontinuous with $g(0)=0$.
\end{assumption}
The functions $g$ satisfying \cref{ass:running_cost} include, for example, the indicator functions of the convex control constraints or the sparsity-promoting terms. To find stabilizing controls for an initial value $y_0\in H$, we study the infinite-horizon problem
\begin{equation}
    \label{eqn:infhorizonproblem}
    \tag{$\text{OP}_{\infty}(y_0)$}
    \min J_\infty(\bu;0,y_0)\quad\text{subject to (s.t.)}\quad\bu\in\mU_\infty(0).
\end{equation}
To address \eqref{eqn:infhorizonproblem}, we employ a receding-horizon scheme, approximating the infinite-horizon problem by concatenating finite-horizon optimal control problems with prediction horizon $T>0$ of the form
\begin{equation}
\label{eqn:fin_hor_ocp}
\tag{$\text{OP}_{T}(\tint,\yint)$}
\min J_T(\bu;\tint,\yint)\quad\text{s.t.}\quad \bu\in\mU_T(\tint).
\end{equation}
Under \cref{ass:pde,ass:running_cost}, the direct method in the calculus of variations ensures that \eqref{eqn:fin_hor_ocp} admits a unique solution $(\bar \bu(\tint,\yint), \bar y(\tint,\yint))$, see, e.g. \cite{hinze2008optimization}. For brevity, we set
\begin{equation}\nonumber
    \bar \ell(t;\tint,\yint)\coloneqq \ell( \bar y(t;\tint,\yint)),\bar \bu(t;\tint,\yint))\quad\text{for }t\in(\tint,\tint+T).
\end{equation}
For a \emph{sampling time} $\delta>0$, we define grid points $t_k = k\delta$ for $k\in\N_0$. At each $t_k$, problem \eqref{eqn:fin_hor_ocp} is solved with $\tint = t_k$, and the solution is applied up to $t_{k+1} = t_k+\delta$, yielding a new initial value $\yint$ for the next step. This procedure is summarized in \cref{algo:FOMRHC}.
\begin{algorithm}[htbp!]
\caption{RHC($\delta,T$)}\label{algo:FOMRHC}
\begin{algorithmic}[1]
\REQUIRE{Final time $T_{\infty} \in \mathbb{R}_{\geq 0} \cup \{ \infty\}$, sampling time $\delta>0$,  prediction horizon $T\geq \delta$, initial value $y_0\in H$;}
\ENSURE{RHC~$\mathbf{u}_{rh}$, non-decreasing sequence~$\{ t_k \}_{k\in\mathbb{N}}$.}
\STATE Set~$(\tint,\yint)\coloneqq(0,y_0)$, $y_{rh}(\tint)\coloneqq y_0$, $k\coloneqq 0$, and $t_0\coloneqq 0$;
\WHILE{$ \tint <T_\infty$}
\STATE Find the optimal solution $(\bar \bu(\cdot\,;\tint,\yint)$, $\bar y(\cdot\,;\tint,\yint))$
by solving~\eqref{eqn:fin_hor_ocp};
\STATE For $\tau\in[\tint,\tint+\delta)$ set $y_{rh}(\tau)\coloneqq  \bar y(\tau;\tint,\yint)$, and $\mathbf{u}_{rh}(\tau)\coloneqq \bar \bu(\tau;\tint,\yint)$;
\STATE Update $k\gets k+1$;  $t_k \gets \tint +\delta$; $(\tint,\yint)\leftarrow(t_k,\bar y(t_k;\tint,\yint))$;
\ENDWHILE
\end{algorithmic}
\end{algorithm}
%
\subsection{Suboptimality and Stability of RHC}
In this section, we address the exponential stabilizability and suboptimality of the RHC obtained from \cref{algo:FOMRHC} for {the} FOM. The suboptimality is measured in terms of the value function, for both the finite- and infinite-horizon cases, as defined below.
\begin{definition}[Value function]
For any $y_0 \in H$, the \emph{infinite-horizon value function} $V_{\infty}: H \to \R_{\geq 0}$ is defined by
\begin{equation*}
    V_{\infty}(y_0)\coloneqq \inf \big\{J_{\infty}(\bu;0,y_0)\,\big|\,\bu\in \mU_\infty(0)\big\}.
\end{equation*}
Similarly, for every $(T,\tint,\yint)\in{\R_{>0}\times\R_{\ge0}} \times H$, the \emph{finite-horizon value function} $V_{T}:\R_{\ge0} \times H \to \R_{\ge0}$ is given as
\begin{equation*}
    V_{T}(\tint,\yint)\coloneqq  \min \big\{J_{T}(\bu;\tint,\yint)\,\big|\,\bu\in \mU_T(\tint)\big\}.
\end{equation*}
\end{definition}
The following assumption is essential for establishing the stability and suboptimality of RHC. It requires that the optimal costs for both the finite- and infinite-horizon problems are uniformly bounded with respect to the $H$-norm of the initial condition.
\begin{assumption}\label{ass:1}
For every $T>0$, $V_T$ is \emph{globally decrescent} with respect to the $H$-norm, that is, there exists a continuous, bounded, and non-decreasing function $\gamma$ with
        \begin{align}\label{e7}
            V_T(\tint,\yint)\leq \gamma(T)\norm{\yint}{H}^2\quad\text{for all }(\tint,\yint)\in\R_{\ge0}\times H.
        \end{align}
\end{assumption}
We begin with a set of auxiliary lemmas that are essential for our main results.
\begin{lemma}\label{ass:2}
    Let \cref{ass:pde,ass:running_cost} be valid. For every $(T,\tint ,\yint) \in \R_{>0}\times \R_{\geq 0} \times  H$ and $\bu \in \mU_T(\tint)$, we have for $y=y(\bu,\tint,\yint)$ solving \eqref{eq:LTV} that
    \begin{align}
        &\norm{y}{C([\tint,\tint+T];H)}^2+\norm{y}{L^2(\tint,\tint+T;V)}^2  \leq C_1 \left(  \norm{\yint}{H}^2 +  J_{T}(\bu; \tint, \yint)  \right)\label{e8},\\
        &\norm{y(\tint+T)}{H}^2 \leq C_2(T) J_{T}(\bu; \tint, \yint)\label{e8a},
    \end{align}
    where $C_1$ and  $C_2$  are independent of $(T, \tint, \yint, \bu)$ and $( \tint,\yint, \bu)$, respectively.
\end{lemma}
\begin{proof}
    Testing \eqref{eq:LTV} with $y(t)$, and  integrating over $(\tint,\tint+T)$ leads with \eqref{eq:weak_coercivity} and Young's inequality to
    \begin{align*}\nonumber
        \norm{y(t)}{H}^2+\eta_V\norm{y}{L^2(\tint,\tint + T;V)}^2 &\leq \norm{y(\tint)}{H}^2+\tfrac{\norm{B}{L^\infty}^2}{\eta_V}\norm{\bu}{\mU_T(\tint)}^2+2\eta_H\norm{y}{L^2(\tint,\tint + T;H)}^2\\
        & \leq\norm{y(\tint)}{H}^2+\max\Big\{\tfrac{2\norm{B}{L^\infty}^2}{\lambda\eta_V}, 4\eta_H \Big\} J_T(\bu;\tint,\yint)
    \end{align*}
    with $|B|_{L^\infty}\coloneqq |B|_{L^\infty(0,T;\mathcal{L}(U,V'))}$. Hence, \eqref{e8} holds with $C_1=\tfrac{\max\{1,\nicefrac{2\norm{B}{L^\infty}^2}{\lambda\eta_V},4\eta_H\}}{ \min\{1,\eta_V\}}$. Turning to \eqref{e8a}, we test \eqref{eq:LTV} with $(\nicefrac{t-\tint}{T})y(t)$ to get f.a.a. $t\in (\tint,\tint+T)$
    \begin{align*}
        \big\langle\dot y(t),\tfrac{t-\tint}{T}y(t)\big\rangle_{V',V}+\Big\langle A(t)\sqrt{\tfrac{t-\tint}{T}}y(t),\sqrt{\tfrac{t-\tint}{T}}y(t)\Big\rangle_{V',V}= \big\langle B(t)\bu(t),{\tfrac{t-\tint}{T}}y(t)\big\rangle_{V',V}.
    \end{align*}
    Integrating over $(\tint,\tint+T)$ and using partial integration for the first term
    \begin{equation*}
        \int_{\tint}^{\tint+T}\big\langle\dot y(t),\tfrac{t-\tint}{T}y(t)\big\rangle_{V',V} = \tfrac{1}{2}\norm{y(\tint+T)}{H}^2-\tfrac{1}{2T} \int_{\tint}^{\tint+T}\norm{y(t)}{H}^2\ \dt.
    \end{equation*}
    Together with \eqref{eq:weak_coercivity}, Young's inequality, and $\sqrt{(\nicefrac{t-\tint}{T})}\leq 1 $, this leads to
    \begin{align*}
        \norm{y(\tint+T)}{H}^2&\leq \tfrac{\norm{B}{L^\infty}^2}{\eta_V}\norm{\bu}{\mU(\tint)}^2+\max\left\{2\eta_H, \tfrac{1}{T}\right\}\norm{y}{L^2(\tint,\tint+T;H)}^2\leq C_2(T) J_T(\bu;\tint,\yint)
    \end{align*}
    with $C_2(T)\coloneqq \max\{\nicefrac{2\norm{B}{L^\infty}^2}{\lambda\eta_V},4\eta_H, \nicefrac{2}{T}\}$.
\end{proof}
\begin{lemma}\label{lem2}
    If \cref{ass:pde,ass:running_cost,ass:1} hold and $T>\delta>0$, then for every  $(\tint,\yint) \in  \mathbb{R}_{\geq 0} \times H$ {it holds
    \begin{align}
        \label{lem2c1}
        &\begin{aligned}
            V_T(\tint+\delta, \bar y(\tint+\delta;\tint, \yint))\leq\int^{\tint+s}_{\tint+\delta} \bar \ell(t;\tint, \yint)\ \dt+\gamma(T+\delta-s)\norm{\bar y(\tint+ s;\tint, \yint)}{H}^2 \quad \text{for } s \in [\delta, T],
        \end{aligned}
        \end{align}
        and
        \begin{align}
        \label{lem2c2}
        &\int^{\tint+T}_{\tint+s} \bar\ell(t;\tint, \yint)\ \dt \leq\gamma(T-s)\norm{\bar y(\tint+s;\tint, \yint)}{H}^2 \quad \text{for }s \in [0, T].
    \end{align}}
\end{lemma}
\begin{proof}
The proof has been given in \cite[Lemma 2.3]{azmi2019hybrid}. 
\end{proof}
\begin{lemma}\label{lem3}
    Suppose that \cref{ass:pde,ass:running_cost,ass:1} hold. Then for $ (\tint,\yint) \in  \mathbb{R}_{\geq 0} \times H$, $T>\delta>0$,
    and the choice of
    \begin{equation*}
    \theta_1=\theta_1(T,\delta) \coloneqq \gamma(T-\delta)C_2(\delta) \quad \text{and}\quad \theta_2=\theta_2(T,\delta) \coloneqq
    \frac{\gamma(T)C_1(C_2(\delta)+\theta_1)}{T-\delta},
    \end{equation*}
    we have the following estimates
    \begin{align}
        \label{e23}
        &\int^{ \tint+T}_{\tint+\delta} \bar \ell(t; \tint, \yint)\ \dt \leq \theta_1 \int^{\tint+\delta}_{\tint} \bar \ell(t; \tint, \yint)\ \dt,\\
        \label{e24}
        &\begin{aligned}
            V_T(\tint+\delta,\bar y(\tint+\delta;\tint, \yint))\quad\leq \int^{\tint+T}_{\tint+\delta} \bar\ell(t;\tint, \yint)\ \dt + \theta_2\int_{\tint}^{\tint+\delta} \bar\ell(t;\tint, \yint)\ \dt.
        \end{aligned}
    \end{align}
\end{lemma}
\begin{proof}
    To verify the inequality \eqref{e23}, we can write by \eqref{lem2c2} that 
    \begin{equation}
        \begin{split}
            \int^{ \tint+T}_{\tint+\delta} \bar \ell(t;\tint, \yint)\ \dt  &\hspace{-1.5mm}   \stackrel{\text{\eqref{lem2c2}}}{\leq} 
            \gamma(T-\delta) \norm{\bar y(\tint+\delta;\tint, \yint)}{H}^2 \\
            &\hspace{-1.5mm}   \stackrel{\text{\eqref{e8a}}}{\leq} \gamma(T-\delta)C_2(\delta) \int_{ \tint}^{\tint+\delta} \bar \ell(t;\tint, \yint)\ \dt,
        \end{split}
    \end{equation}
    concluding \eqref{e23}. Turning to \eqref{e24}, recall that $ \bar y(\cdot\,;\tint, \yint) \in C([\tint,\tint+T];H)$. Hence, there is a $\hat{t}\in [\delta, T]$ such that $\hat{t}= \argmin_{t \in [\delta, T]}  \norm{\bar y(\tint+t ;\tint, \yint)}{H}^2$. By \eqref{lem2c1}, we have using that $\gamma$ is non-decreasing by \cref{ass:1}
    \begin{equation}
        \label{e15}
        \begin{split}
        &V_T( \tint+\delta, \bar y(\tint+\delta; \tint, \yint))\\
        &\hspace{10mm}\stackrel{\text{\eqref{lem2c1}}}{\leq}\int^{\tint+\hat{t}}_{\tint+\delta} \bar \ell(t; \tint, \yint)\ \dt+\gamma(T+\delta-\hat{t}) \norm{ \bar y(\tint+\hat{t}; \tint, \yint)}{H}^2\\
        &\hspace{11.5mm}\leq \int^{\tint+\hat{t}}_{ \tint+\delta} \bar \ell(t;\tint, \yint)\ \dt+\gamma(T)\norm{\bar y(\tint+\hat t;\tint, \yint)}{H}^2\\
        &\hspace{11.5mm}\leq \int^{\tint+T}_{\tint+\delta} \bar \ell(t;\tint, \yint)\ \dt+\tfrac{\gamma(T)}{T-\delta}\norm{\bar y(\cdot;\tint, \yint)}{L^2(\tint+\delta,\tint+T;H)}^2.
        \end{split}
    \end{equation}
    Together with \cref{ass:2}
    \begin{align*}
        &\norm{\bar y(\cdot;\tint, \yint)}{L^2(\tint+\delta,\tint+T;H)}^2\leq \norm{\bar y(\cdot;\tint, \yint)}{{L^2(\tint+\delta,\tint+T;V)}}^2\\
        &\hspace{10mm}\stackrel{\text{\eqref{e8}}}{\leq} C_1 \left(   \norm{\bar y(\tint+\delta ;\tint, \yint)}{H}^2+\int^{\tint+T}_{\tint+\delta}\bar \ell(t;\tint, \yint)\ \dt \right)\\
        &\hspace{10mm}\stackrel{\text{\eqref{e8a}}}{\leq}  C_1  \left( C_2(\delta)\int_{\tint}^{\tint+\delta}\bar \ell(t;\tint, \yint)\ \dt +\int^{\tint+T}_{\tint+\delta}\bar\ell(t;\tint, \yint)\ \dt    \right) \\
        &\hspace{10mm}\stackrel{\text{\eqref{e23}}}{\leq}  C_1\big(C_2(\delta)+\theta_1\big)\int_{\tint}^{\tint+\delta}\bar \ell(t;\tint, \yint)\ \dt,
    \end{align*}
    we can conclude  \eqref{e24}.
\end{proof}
\begin{proposition}\label{pro1}
    Suppose that \cref{ass:pde,ass:running_cost,ass:1} hold and let $\delta>0$ be given. Then there exist  $\bar T>\delta$ and $\alpha \in (0,1)$ such that for every $T\geq \bar T$, the following inequalities hold for all $(\tint, \yint) \in  \R_{\geq 0} \times H$
    \begin{align}
    \label{e20s}
    V_T(\tint+\delta ,\bar y(\tint+\delta;\tint, \yint)) &\leq V_T(\tint, \yint)-\alpha \int_{\tint}^{\tint+\delta} \bar\ell(t;\tint, \yint)\ \dt, \\\label{e20e}
    V_T(\tint+\delta ,\bar y(\tint+\delta;\tint, \yint)) &\leq e^{-\zeta \delta} V_T(\tint, \yint),
    \end{align}
    where $\zeta>0$ depends only on $(\theta_1,\theta_2,\alpha)$. 
\end{proposition}
\begin{proof}
    From the definition of $V_T(\tint, \yint)$ and \eqref{e24}, we obtain
    \begin{align*}
        V_T(\tint+\delta ,\bar y(\tint+\delta;\tint, \yint)) - V_T(\tint, \yint)
        \le(\theta_2-1)\int^{\tint+\delta}_{\tint}\bar \ell(t;\tint,  \yint)\ \dt,
    \end{align*}
    where $\theta_2=\theta_2(T,\delta)$ is defined in \cref{lem3}. Due to the boundedness of $\gamma$ in \cref{ass:1}, we have for a fixed $\delta>0$
    \begin{equation}
    \label{e3}
     \alpha(T,\delta) := 1-\theta_2(T,\delta) \to 1 \text{ as } T\to \infty,
    \end{equation}
    and there exist $\bar T>\delta$ and $\alpha(\bar T,\delta) \in (0,1)$ such that $1-\theta_2(T,\delta)\geq \alpha(\bar T,\delta)$ for all
    $T\geq \bar T$. This implies \eqref{e20s}.
    Now, we turn to the verification of \eqref{e20e}. Using \eqref{e23} and \eqref{e24} we have
    \begin{equation}
    \label{e44}
    \begin{split}
    V_T( \tint+\delta,\bar y( \tint+\delta; \tint, \yint)) 
                                        \leq (\theta_1+\theta_2) \int^{\tint+\delta}_{\tint} \bar \ell(t;\tint,\yint)\ \dt.
    \end{split}
    \end{equation}
    Together with  \eqref{e20s}, we obtain 
    \begin{equation*}
    \begin{split}
    V_T(\tint+\delta ,\bar y( \tint+\delta; \tint, \yint)) - V_T( \tint,  \yint)  \leq  \tfrac{-\alpha}{\theta_1+\theta_2}V_T(\tint+\delta , \bar y( \tint+\delta; \tint,  \yint)). 
    \end{split}
    \end{equation*}
    Thus, by defining $\eta\coloneqq(1+\nicefrac{\alpha}{\theta_1+\theta_2})^{-1}\in(0,1)$, we can write
    \begin{equation*}
    V_T( \tint+\delta ,\bar y( \tint+\delta; \tint, \yint)) \leq  \eta V_T(\tint,\yint)
    \end{equation*}
    and,  as a consequence,   \eqref{e20e}  follows by setting  $\zeta\coloneqq\nicefrac{\lvert  \ln\eta \lvert }{\delta}$.
\end{proof}
In the next theorem, we present the main result of this section, namely the exponential stability and suboptimality of Algorithm \ref{algo:FOMRHC}. This result relies on property \eqref{e7} from \cref{ass:1}, conditions \eqref{e8} and \eqref{e8a} from \cref{ass:2}, as well as the well-posedness of \eqref{eqn:fin_hor_ocp} for every pair $(\tint,\yint)\in \R_{\ge 0}\times H$.

\begin{theorem}[Suboptimality and exponential stability]\label{theorem:stabilityRHC}
Let \cref{ass:pde,ass:running_cost,ass:1} hold. For a sampling time $\delta>0$, there exist $\bar T > \delta$ and $\alpha \in (0,1)$ such that, for every fixed prediction horizon $T \geq \bar T$, the RHC $\bu_{rh}$ obtained from \cref{algo:FOMRHC} is suboptimal and exponentially stabilizing for any $y_0 \in H$. That is,
\begin{align}
\label{ed27}
    V_{\infty}(y_0) \leq J_{\infty}(\bu_{rh};0,y_0)\leq \tfrac{1}{\alpha}V_{T}(0,y_0) \leq \tfrac{1}{\alpha} V_{\infty}(y_0),\\
    \label{ed28}
    \norm{y_{rh}(t)}{H}^2 \leq  C_{rh}e^{-\zeta t} \norm{y_0}{H}^2 \quad  \text{ for }  t\geq 0,
\end{align}
where the positive numbers $\zeta$ and $C_{rh}$ depend on $(\alpha,\delta,T)$, but are independent of $y_0$.
\end{theorem}
\begin{proof}
    Consider the sampling instances $t_k=k\delta$ for $k\in \N_0$ from \cref{algo:FOMRHC}. The first inequality in \eqref{ed27} is trivial. For the second, we sum up \eqref{e20s} from \cref{pro1} with $\tint = t_k$, $\tint+\delta = t_{k+1}$ for $k=0,\ldots,k'$ to obtain
    \begin{equation}
    \nonumber
        \begin{aligned}
            \alpha J_{t_k'}(\bu_{rh};0,y_0)\leq  \sum\limits_{k=0}^{k'}V_T(t_k,y_{rh}(t_k))-V_T(t_{k+1},y_{rh}(t_{k+1}))\leq V_T(0,y_0)\leq V_\infty(y_0),
        \end{aligned}
    \end{equation}
    since $V_T\geq 0$ due to \eqref{eq:ell_property}. Letting $k'\to\infty$, we obtain \eqref{ed27}.
    We also have 
    \begin{equation}
        \nonumber 
        \begin{aligned}
            \norm{y_{rh}(t_{k+1})}{H}^2 &\stackrel{\text{\eqref{e8a}}}{\leq} C_2(\delta) V_T( t_{k} , y_{rh}(t_{k}))\stackrel{\text{\eqref{e20e}}}{\leq} C_2(\delta) e^{-\zeta t_{k}}  V_T(0, y_0)\\
            &\stackrel{\text{\eqref{e7}}}{\leq} C_2(\delta) \gamma(T) e^{-\zeta t_{k}} \norm{y_0}{H}^2.
        \end{aligned}
    \end{equation}
    Furthermore, setting $C_H =\nicefrac{ C_2(\delta) \gamma(T)}{\eta}$ with $\eta = e^{-\delta \zeta}$ we have 
    \begin{equation}
        \label{e93}
        \begin{split}
        \norm{y_{rh}(t_{k+1})}{H}^2 &\leq C_2(\delta) \gamma(T) e^{-\zeta t_{k}}  \norm{y_0}{H}^2 =  C_He^{-\zeta t_{k+1}} \norm{y_0}{H}^2 \quad \text{ for }  k \in \N_0. 
        \end{split}
    \end{equation}
    Moreover, for every $t >0$ there exists a $k \in \mathbb{N}$ such that $t \in [t_k, t_{k+1}]$. For $t \in [t_k, t_{k+1}]$,
    \begin{equation}
        \label{eq:exp_cont_in_t}
        \begin{split}
            \norm{y_{rh}(t)}{H}^2 &\stackrel{\text{\eqref{e8}}}{\leq} C_1\left(\norm{y_{rh}(t_k)}{H}^2+V_T(t_{k},y_{rh}(t_{k}))\right)\stackrel{\text{\eqref{e7}}}{\leq}C_1(1+\gamma(T))\norm{y_{rh}(t_k)}{H}^2\\
            &\hspace{-0.5mm} \stackrel{\text{\eqref{e93}}}{\leq} C_1C_H(1+\gamma(T)) e^{-\zeta t_k} \norm{y_0}{H}^2\leq  C_1C_H(1+\gamma(T))\eta^{-1}e^{-\zeta t_{k+1}}\norm{y_0}{H}^2\\
            &\hspace{2mm}\leq  C_1C_H(1+\gamma(T))\eta^{-1}e^{-\zeta t}\norm{y_0}{H}^2,
        \end{split}
    \end{equation}
    and therefore by setting $C_{rh} \coloneqq C_1C_H(1+\gamma(T))\eta^{-1}$
    we directly infer \eqref{ed28}.
\end{proof}
\begin{remark}
    For a fixed $\delta > 0$ we infer from \eqref{e3} that $\lim_{T \to \infty} \alpha(T) = 1$ and $\lim_{T \to \delta} \alpha(T) = -\infty$. That is, RHC is \emph{asymptotically optimal}. Moreover for fixed $\delta>0$, the constants $(C_{rh},\zeta)$ can be bounded independently of $T \geq \bar T$. Further, for a fixed $T \geq \bar T$, we can also see in the proof of \cref{ass:2} that $C_2(\delta)\to \infty$ as $\delta \to 0$, which implies $\theta_2(T,\delta)\to \infty$ and $\alpha = 1-\theta_2(T,\delta)\to -\infty$ as $\delta \to 0$.
\end{remark}
The goal of the paper is now to investigate conditions under which a similar result holds for the reduced counterpart of \cref{algo:FOMRHC}. 
For later use, we state the optimality condition of \eqref{eqn:fin_hor_ocp}.
\begin{remark}[Finite-horizon optimality condition]
    \label{rem:optcond}
    Let $(T,\tint,\yint)\in{\R_{>0}\times\R_{\ge0}} \times H$. Given $y\in L^2(\tint,\tint+T;H)$, consider the \emph{adjoint equation}
    \begin{align}
        \label{eq:FOM_opsys}
            -\ddt p(t)+A'(t)p(t)= y(t), \quad t \in  (\tint,\tint+T),\quad 
            p(\tint+T)=0.
    \end{align}
    By \cref{ass:pde}, there exists a unique solution $p=p(y)\in \mY_T(\tint)$. Moreover, $\bar \bu = \bar\bu(\tint,\yint)$ is optimal for \eqref{eqn:fin_hor_ocp}, if and only if there exist $\bar y,\bar p\in\mY_T(\tint)$ with $\bar y=\bar y(\bar \bu,\tint,\yint)$ solving \eqref{eq:LTV} for $\bu=\bar \bu$, $\bar p=\bar p(\bar y)$ solving \eqref{eq:FOM_opsys} for $y =\bar y $, and
    \begin{align}
        \label{eq:FOM_opcond}
        \int_{\tint}^{\tint+T}{\langle B'(t)\bar p(t)+ \lambda \bar \bu(t), \bu(t)-\bar \bu(t)\rangle}_U\,\dt \geq g_T(\bar \bu)-g_T(\bu)\quad\text{for all }\bu\in\mU_T(\tint).
    \end{align}
    Here, we have set $g_T(\bu;\tint)\coloneqq \int_{\tint}^{\tint+T}g(\bu(t))\ \dt$ and in \eqref{eq:FOM_opsys}, \eqref{eq:FOM_opcond}, the operators $A'(t)\in\calL(V,V')$, $B'(t)\in\calL(V,U)$ denote the adjoint of $A(t)$ and $B(t)$, respectively. Defining $x=(y,\bu,p)\in\mX_T(\tint)\coloneqq \mY_T(\tint)\times\mU_T(\tint)\times\mY_T(\tint)$ and the Lagrangian of the cost function's smooth part as
    \begin{align*}
        L(x;\yint)&\coloneqq \int_{\tint}^{\tint+T}\left( \tfrac{1}{2}\norm{y(t)}{H}^2+\tfrac{\lambda}{2}\norm{\bu(t)}{U}^2+\langle B(t)\bu(t)-A(t)y(t)-\partial_t y(t), p(t) \rangle_{V',V}\,\right)\dt\\
        &\quad+ \langle \yint-y(\tint),p(\tint)\rangle_H,
    \end{align*}
    we can express the optimality condition compactly as
    \begin{subequations}\label{eq:FOML}
        \begin{align}
            L'_y(\bar x;\yint)(y) &= 0&&\hspace{-25mm}\text{for all }y\in\mY_T(\tint),
            \label{eq:FOML1}\\
            L'_p(\bar x;\yint)(p) &= 0&&\hspace{-25mm}\text{for all }p\in\mY_T(\tint),
            \label{eq:FOML2}\\
            L'_{\bu}(\bar x;\yint)(\bu-\bar \bu) &\geq g_T(\bar \bu)-g_T(\bu)&&\hspace{-25mm}\text{for all }\bu\in\mU_T(\tint).
            \label{eq:FOML3}\
        \end{align}
    \end{subequations}
\end{remark}
%
\section{Reduced-order modeling for the finite-horizon problem}\label{sec:ROM}
%
\cref{algo:FOMRHC} represents a multi-query scenario for the finite-horizon FOM open-loop problem \eqref{eqn:fin_hor_ocp}. In this section, we introduce a cheap-to-compute reduced version \eqref{eqn:reducedStabilizationProblem}, derive corresponding error estimates, and analyze the properties of the ROM in \cref{subsec:apost} and \cref{subsec:ROMproperties}, respectively.

Let $T>0$ be finite throughout this section, and let $V_r\subset V$ be a finite-dimensional linear (reduced-order) subspace. By Galerkin projection, the reduced-order solution $\reduce{y}(t)\in V_r$ satisfies for $(\tint,\tildeyint)\in\R_{\ge0}\times H$
\begin{equation}
    \label{eq:reducedLTV}
    \tag{$\text{ROM}^r_{T}(\tint, \tildeyint)$}
     \ddt {\reduce {y}}(t)+A(t){{\reduce {y}}}(t)= {B}(t)\bu(t)\ \text{in } V_r', \ t \in  ( \tint,\tint+T),\ 
           \reduce{y}(\tint)
           =\Pi^H_{V_r} \tildeyint,
\end{equation}
where $\Pi^H_{V_r}:H\to V_r$ is the $H$-orthogonal projection onto $V_r$, characterized as unique solution to $\langle\Pi^H_{V_r}\tildeyint,v\rangle_H=\langle \tildeyint ,v\rangle_H$ for all $v\in V_r$. Note that we allow for $\tildeyint \neq \yint$. 
We call \eqref{eq:reducedLTV} the \emph{reduced-order model} (ROM) and by \cref{ass:pde}, there exists a unique solution $\reduce{y}=y^r(\bu,\tint,\tildeyint)\in H^1(\tint,\tint+T;V_r)\hookrightarrow \mY_T^r(\tint)$ for all $(\bu,\tildeyint)\in \mU_T(\tint)\times H$.
In addition, we can introduce the reduced finite-horizon problem
\begin{equation}
    \label{eqn:reducedStabilizationProblem}
    \tag{$\text{OP}^r_{T}(\tint,\tildeyint)$}
    \min\limits_{\bu\in\mU_T(\tint)} J_T^r(\bu;\tint,\tildeyint)\coloneqq\int^{\tint+T}_{\tint} \ell(y^r(t;\bu,\tint,\tildeyint),\bu(t))\ \dt
\end{equation}
and the reduced finite-horizon value function $V^r_{T}: \mathbb{R}_{\geq 0} \times H \to \mathbb{R}_{\geq 0}$ 
\begin{equation*}
    V^r_T(\tint,\tildeyint)\coloneqq \inf\big\{J_T^r(\bu;\tint,\tildeyint)\,\big|\,\bu\in \mU_T(\tint)\big\}.
\end{equation*}
\begin{remark}[Reduced finite-horizon optimality condition]
    Given data $\tilde y\in L^2(\tint,\tint+T;H)$, we introduce the reduced adjoint system 
    \begin{align}
        \label{eq:ROM_opsys}
            -\ddt {\reduce {p}}(t)+A'(t){\reduce {p}}(t)= { \tilde{y}}(t)\ \text{in } V_r', \quad t \in  ( \tint,\tint+T),\quad 
            {\reduce {p}}(\tint+T)=0,
    \end{align}
    with the unique solution $p^r=p^r(\tilde y)\in H^1(\tint,\tint+T;V_r)$.
    Also, \eqref{eqn:reducedStabilizationProblem} is uniquely solvable and the unique solution is given by ${\reduce {\bar \bu}} = {\reduce {\bar \bu}}(\tint,\tildeyint)$, if and only if there exists ${\reduce {\bar y}},{\reduce {\bar p}}\in\mY^r_T(\tint)$ with ${\reduce {\bar y}}={\reduce {\bar y}}(\bar\bu^r,\tint,\tildeyint)$ solving \eqref{eq:reducedLTV} for $\bu={\reduce {\bar \bu}} $, ${\reduce {\bar p}}={\reduce {\bar p}}(\bar y^r)$ solving \eqref{eq:ROM_opsys} for $\tilde y={{\bar y^r}} $, and
    \begin{align}\label{eq:ROM_opcond}
        \int_{\tint}^{\tint+T}{\langle B'(t){\reduce {\bar p}}(t)+\lambda {\reduce {\bar\bu}}(t),\bu(t)-{\reduce{\bar \bu}(t)}\rangle}_U\,\dt\geq g_T(\reduce{\bar\bu})-g_T(\bu)\text{ for all }\bu\in\mU_T(\tint).
    \end{align}
    For $\reduce{\bar x}=(\reduce{\bar y},\reduce{\bar\bu},\reduce{\bar p})\in{\mX}^r_T(\tint)\coloneqq\mY^r_T(\tint)\times\mU_T(\tint)\times\mY^r_T(\tint)$ we get the optimality condition
    \begin{subequations}
        \label{eq:ROML}
        \begin{align}
            L'_y(\reduce{\bar x};\tildeyint)(y^r) &= 0&&\text{for all }y^r\in\mY^r_T(\tint),
            \label{eq:ROML1}\\
            L'_p(\bar x^r;\tildeyint)(p^r) &= 0&&\text{for all }p^r\in\mY^r_T(\tint),
            \label{eq:ROML2}\\
            L'_\bu(\bar x^r;\tildeyint)(\bu-\bar \bu^r) &\geq g_T(\bar \bu^r)-g_T(\bu)&&\text{for all }\bu\in\mU_T(\tint).
            \label{eq:ROML3}
        \end{align}
    \end{subequations}
\end{remark}
%
\subsection{A posteriori error estimation for the finite-horizon problem}\label{subsec:apost}
%
In this section, we present \emph{a posteriori} error estimates that serve to quantify the performance of the reduced RHC algorithm. In \cref{subsec:APOS_stat}, we derive estimators for the state, adjoint state, and optimal control, while in \cref{subsec:APOST_val}, we establish error estimates for the cost and value functions.
%
\subsubsection{State, adjoint state, and optimal control estimates}\label{subsec:APOS_stat}
%
Let the initial values of the \eqref{eq:LTV} and the \eqref{eq:reducedLTV} satisfy 
\begin{equation}\label{eq:perturb_init}
    \norm{ \yint-\tildeyint}{H} \leq \Delta_{\yint} \quad \text{for } \Delta_{\yint}\geq 0.
\end{equation}
In the following lemmas, we establish the corresponding error estimators.
\begin{lemma}[State a posteriori estimator]
    \label{Lemma:ResBasStateError}
    Let \cref{ass:pde} be valid and
    let $y=y(\bu,\tint,\yint)\in\mY_T(\tint)$ and $y^r=y^r(\bu^r ,\tint,\tildeyint)\in\mY^r_T(\tint)$ be the solution of \eqref{eq:LTV} and \eqref{eq:reducedLTV} for $\yint,\tildeyint \in H$ and $\bu,\bu^r\in\mU_T(\tint)$, respectively. Assume the data satisfies \eqref{eq:perturb_init} and ${\norm{\bu-\reduce\bu}{\mU_T(\tint)}}\leq\Delta_\bu$
    for $\Delta_\bu\geq 0$. Define the state error and residual as
    \begin{equation}\nonumber
        e_y\coloneqq y-\reduce y, \quad \Res_y(y^r,\bu^r)(t)\coloneqq B(t)\bu^r(t)-A(t)y^r(t)-\partial_t y^r(t) \in V' \text{ for }t\in(0,T).
    \end{equation}
    Then, we have the \emph{a posteriori} error bound for $t\in [\tint, \tint+T]$
    \begin{align}\label{eq:state_apost_CHL2V}
        &\norm{e_y(t)}{H}^2+\norm{e_y}{L^2(\tint,t;V)}^2\leq \Delta_y^2(t;\Delta_\bu, \Delta_{\yint}, y^r, \bu^r)
    \end{align}
    with
    \begin{align*}
        \Delta^2_{y}(t)&\coloneqq C_{1,y}(t)\Big(\norm{e^{-\eta_H (\cdot-\tint) }B}{L^\infty(\tint,t;\calL(U,V'))}^2{\Delta_\bu^2} + \norm{e^{-\eta_H (\cdot-\tint) }\Res_y(\reduce y, \reduce \bu)}{L^2(\tint,t;V')}^2\Big)\\
        &\hspace{1mm}\quad+ C_{2,y}(t){\big( \Delta_{\yint} + \norm{\tildeyint - \Pi^H_{V_r}\tildeyint }{H} \big)^2}
    \end{align*}
    for the constants $C_{1,y}(t)\coloneqq \nicefrac{2e^{2\eta_H(t-\tint)}}{\min\{1,\eta_V\}\eta_V}$, $C_{2,y}(t)\coloneqq \nicefrac{e^{2\eta_H(t-\tint)}}{\min\{1,\eta_V\}}$. If $t=\tint+T$, we simply write $\Delta_y^2(\Delta_\bu, \Delta_{\yint}, y^r, \bu^r)$ instead of $\Delta_y^2(\tint+T;\Delta_\bu, \Delta_{\yint}, y^r, \bu^r)$.
\end{lemma}
\begin{proof}
    Defining $e_u= \bu-\reduce\bu$, we obtain the following error equation
    \begin{equation}
    \label{eq:error_eq}
    \begin{cases}
    \ddt e_y(t)=-{A}(t)e_y(t)+{B}(t)e_{\bu}(t)+ \Res_y(\reduce y, \reduce \bu)(t) &\text{in } V'\text{ for } t \in ( \tint,\tint+T),\\
    e_y(\tint)= \yint- \Pi^H_{V_r}\tildeyint. &
    \end{cases}
    \end{equation}
    Let $t\in (\tint,\tint+T)$. Testing with $e_y(s)\in V$ f.a.a. $s\in (\tint,t$) and integrating over $(\tint,t)$ leads to
    \begin{align*}
        &\tfrac{1}{2} \norm{e_y(t)}{H}^2 + \int_{\tint}^{t}{\langle A(s) e_y(s), e_y(s) \rangle}_{V',V} \ds=  \tfrac{1}{2} \norm{e_y(\tint)}{H}^2 +\int_{\tint}^{t} \langle B(s)e_u(s) + \Res_y(\reduce y, \reduce \bu)(s), e_y(s)\rangle_{V',V} \ds .
    \end{align*}
    First, we assume $A(s)$ to be coercive, that is, $\eta_H=0$ in \cref{ass:pde}. 
    This results in
    \begin{align*}
        &\tfrac{1}{2}\norm{e_y(t)}{H}^2 + \eta_V \norm{e_y}{L^2(\tint,t;V)}^2\leq   \tfrac{1}{2}\norm{e_y(\tint)}{H}^2 + \int_{\tint}^{t}\tfrac{1}{2\varepsilon} \norm{ B(s)e_u(s) + \Res_y(\reduce y, \reduce \bu)(s)}{V'}^2 +\tfrac{\varepsilon}{2}{e_y(s)}_V^2 \ds
    \end{align*}
    for $\varepsilon\in (0,2\eta_V]$ by Young's inequality. Therefore we have
    \begin{align*}
       &\tfrac{1}{2}\norm{e_y(t)}{H}^2 +\big(\eta_V-\nicefrac{\varepsilon}{2}\big) \norm{e_y}{L^2(\tint,t;V)}^2\le\tfrac{1}{2}\norm{e_y(\tint)}{H}^2 + \tfrac{1}{2\varepsilon} \norm{ Be_u +  \Res_y(\reduce y, \reduce \bu)}{L^2(\tint,t;V')}^2\\
        &\qquad\leq  \tfrac{1}{2}\norm{e_y(\tint)}{H}^2 + \tfrac{1}{\varepsilon} \big(\norm{B}{L^\infty(\tint,t;\calL(U,V'))}^2\Delta_u^2 +\norm{\Res_y(\reduce y, \reduce \bu)}{L^2(\tint,t;V')}^2\big).
    \end{align*}
    The case $\eta_H>0$ follows by using the standard transformation $e_v(t)=e^{-\eta_H (t-\tint)}e_y(t)$. The transformed error equation for $e_v$ is given by the coercive differential operator $\tilde A(t)\coloneqq A(t)+\eta_H$, which, as before, leads to the scaled estimate
    \begin{align}\nonumber
        &\tfrac{1}{2}{\big|e^{-\eta_H (t-\tint)}e_y(t)\big|}_{H}^2+\big(\eta_V-\nicefrac{\varepsilon}{2}\big){\big|e^{-\eta_H (\cdot-\tint) }e_y\big|}_{L^2(\tint,t;V)}^2\\
        &\leq \tfrac{1}{2}\norm{e_y(\tint)}{H}^2 + \tfrac{1}{\varepsilon}\left({\big|e^{-\eta_H (\cdot-\tint) }B\big|}_{L^\infty(\tint,t;\calL(U,V'))}^2\Delta_u^2+ {\big|e^{-\eta_H (\cdot-\tint) }\Res_y(\reduce y, \reduce \bu)\big|}_{L^2(\tint,t;V')}^2\right). 
        \nonumber 
    \end{align}
    Hence, we also find
    \begin{align*}
        &\tfrac{1}{2}{\big|}e_y(t)\big|_H^2+\big(\eta_V-\nicefrac{\varepsilon}{2}\big)\norm{e_y}{L^2(\tint,t;V)}^2\leq \tfrac{e^{2\eta_H (t-\tint)}}{2}\norm{e_y(\tint)}{H}^2\\
        &+ \tfrac{e^{2\eta_H (t-\tint)}}{\varepsilon}\left( {\big|e^{-\eta_H (\cdot-\tint) }B\big|}_{L^\infty(\tint,t;\calL(U,V'))}^2\Delta_u^2 +{\big|e^{-\eta_H (\cdot-\tint) }\Res_y(\reduce y, \reduce \bu)\big|}_{L^2(\tint,t;V')}^2\right).
    \end{align*}
    For the error in the initial condition, we add and subtract $\tildeyint$, and use \eqref{eq:perturb_init}.
    Choosing $\varepsilon= \eta_V$ and $t=\tint+T$ leads to \eqref{eq:state_apost_CHL2V}. 
\end{proof}
\begin{lemma}[Adjoint state a posteriori estimator]
    \label{lem:ResAdjest}
    Suppose that \cref{ass:pde} holds. Let $p=p(y)\in\mY_T(\tint)$ and $p^r=p^r(\tilde y)\in\mY^r_T(\tint)$ be the solutions of \eqref{eq:FOM_opsys} and \eqref{eq:ROM_opsys} for $y,\tilde y\in L^2(\tint,\tint+T;H)$, respectively. Assume that $\norm{y-\tilde y}{L^2(\tint,\tint+T;H)} \leq \Delta_y $
    for $\Delta_y\geq 0$. Define the adjoint state error and residual as
    \begin{equation*}
        e_p\coloneqq p-p^r, \quad \Res_p(\tilde y, p^r)(t)\coloneqq\tilde y(t)-A'(t)p^r(t)+\dot p^r(t)\in V'.
    \end{equation*}
    Then, we have the \emph{a posteriori} error bound
    \begin{align}\label{eq:adjstate_apost_CHL2V}
        &\norm{e_p(\tint)}{H}^2+\norm{e_p}{L^2(\tint,\tint+T;V)}^2\leq \Delta_p^2(\Delta_y,\tilde y, p^r)
    \end{align}
    with
    \begin{align}\label{eq:adjstate_apost_CHL2V_struc}
        \Delta_{p}^2{(\Delta_y,\tilde y, p^r)}  \coloneqq & C_{p}\Big(\Delta_y^2 + \norm{e^{\eta_H (\cdot-\tint-T) }\Res_p(\tilde y, \reduce p)}{L^2(\tint,\tint+T;V')}^2\Big)
    \end{align}
    and $C_{p}\coloneqq\nicefrac{2e^{2\eta_H T}}{\min(1,\eta_V)\eta_V}$. Further, we have the improved estimate
    \begin{align}\label{eq:adjstate_apost_Ht}
        &\norm{e_p(\tint)}{H}^2\leq \tfrac{1}{2} \Delta_{p}^2 {(\Delta_y,\tilde y, p^r)}.
    \end{align}
\end{lemma}
\begin{proof}
    The claim follows by using similar arguments as in the proof of \cref{Lemma:ResBasStateError} and the transformation $e_v(t)=e^{\eta_H(t-\tint-T)}e_p(t)$.
\end{proof}
\begin{theorem}[Optimal control estimator]\label{theo:oc_est}
    Let \cref{ass:pde,ass:running_cost} be valid and let $\bar x=(\bar y, \bar \bu, \bar p)\in\mX_T(\tint),\bar x^r=(\bar y^r, \bar \bu^r, \bar p^r)\in\mX^r_T(\tint)$ be the solutions of the FOM \eqref{eq:FOML} and ROM \eqref{eq:ROML} optimality systems for $\yint,\tildeyint\in H$, respectively. Then, it holds
    \begin{align}\label{eq:opt_cont_est}
        \hspace{-1.95mm}\norm{\bar\bu-\bar\bu^r}{\mU_T(\tint)}^2\leq& \bar \Delta^2_{ \bu}(\bar x^r,\Delta_{\yint})\coloneq \tfrac{\norm{B}{L^\infty}^2}{\lambda^2} \Delta_p^2(0,\bar y^r,\bar p^r) + \tfrac{1}{\lambda}\Delta_y^2(0,0,\bar y^r,\bar \bu^r)+\tfrac{
        C_\bu}{\lambda}\Delta^2_{\yint},\\ 
        \label{eq:opt_out_est}
         \hspace{-1.95mm}\norm{\bar y-\bar y^r}{L^2(H)}^2\leq& \bar \Delta^2_{y,H}(\bar x^r,\Delta_{\yint})\coloneq \tfrac{\norm{B}{L^\infty}^2}{{2}\lambda}  \Delta_p^2(0,\bar y^r,\bar p^r) + 2\Delta_y^2(0,0,\bar y^r,\bar \bu^r)+2
        C_\bu\Delta^2_{\yint}
    \end{align}
    for $C_\bu = \nicefrac{e^{2\eta_HT}}{2\eta_V}$
    and $\Delta_{y}, \Delta_{p}$ as in \cref{Lemma:ResBasStateError,lem:ResAdjest}, respectively.
\end{theorem}
\begin{proof}
    In the proof of \cite[Theorem 3.7]{kartmann2024certifiedmodelpredictivecontrol} the following inequality was proven
    \begin{align*}
        &(\lambda-\tfrac{\varepsilon_1}{2})\norm{\bar \bu-\bar \bu^r}{\mU_T(\tint)}^2+ (1-\tfrac{\varepsilon_2}{2}-\tfrac{\varepsilon_3C_\bu}{2})\norm{\bar y-\bar y^r}{\mU_T(\tint)}^2\\
       & \leq  \tfrac{\norm{B}{L^\infty}^2}{2\varepsilon_1} \norm{p(\bar y^r)-\bar p^r}{L^2(\tint,\tint+T;V)}^2 + \tfrac{1}{2\varepsilon_2}\norm{y(\bar \bu^r,\tildeyint)-\bar y^r}{L^2(\tint,\tint+T;H)}^2+\tfrac{1}{2\varepsilon_3}\Delta_{\yint}^2
    \end{align*}
    for $\varepsilon_1, \varepsilon_2, \varepsilon_3 >0$.
    By \eqref{eq:adjstate_apost_CHL2V} in \cref{lem:ResAdjest} for $\tilde y=y=\bar y^r$, we estimate
    \begin{align*}
        \norm{p(\bar y^r)-\bar p^r}{L^2(\tint,\tint+T;V)}^2=\norm{p(\bar y^r)-p^r(\bar y^r)}{L^2(\tint,\tint+T;V)}^2\leq \Delta_p^{{2}}(0,\bar p^r,\bar y^r),
    \end{align*}
    and by \eqref{eq:state_apost_CHL2V} in \cref{Lemma:ResBasStateError} for $\bu^r=\bu=\bar \bu^r$ and $\Delta_{\yint}=0$, we have
    \begin{align*}
        \norm{y(\bar \bu^r,\tildeyint)-\bar y^r}{{L^2(\tint,\tint+T;V)}}^2= \norm{y(\bar \bu^r, \tildeyint)-y^r(\bar \bu^r, \tildeyint)}{{L^2(\tint,\tint+T;V)}}^2  \leq \Delta_y^{{2}}(0,0,\bar y^r,\bar \bu^r).
    \end{align*}
    Now, choosing $\varepsilon_1 =\lambda $, $\varepsilon_2=1 $, $\varepsilon_3= \nicefrac{1}{C_\bu}$ implies \eqref{eq:opt_cont_est} and the choice $\varepsilon_1 =2\lambda $, $\varepsilon_2=\nicefrac{1}{2} $, $\varepsilon_3= \nicefrac{1}{2C_\bu}$ implies \eqref{eq:opt_out_est}.
\end{proof}
\begin{remark}
    In \cref{Lemma:ResBasStateError}, the constants $(C_{1,y}, C_{1,y}(t))$ can be improved by the factor $\nicefrac{1}{2}$ if $\bu=\bu^r$, i.e., $\Delta_\bu=0$. Similarly, the constant $C_p$ in \cref{lem:ResAdjest} can be reduced by the factor $\nicefrac{1}{2}$ if $y=\tilde y$, i.e., $\Delta_y=0$.
\end{remark}

\subsubsection{Value and cost function estimates}\label{subsec:APOST_val}
%
The following result extends the value function estimator from \cite{ran,vexler2008adaptive} to the general convex regularization function $g_T$ from \cref{rem:optcond} satisfying \cref{ass:running_cost} and perturbed initial values with \eqref{eq:perturb_init}.
\begin{theorem}[Optimal value function error representation]
    \label{lem:errorest_optimalval_diff_initguess_prepare}
    Let \cref{ass:pde,ass:running_cost} be valid and let $\bar x=(\bar y, \bar \bu, \bar p)\in\mX_T(\tint),\bar x^r=(\bar y^r, \bar \bu^r, \bar p^r)\in\mX^r_T(\tint)$ be the solutions of the FOM \eqref{eq:FOML} and ROM \eqref{eq:ROML} optimality systems for $\yint,\tildeyint\in H$, respectively. Then, we can bound the error in the optimal value function as
\begin{align}
    \label{eq:errorest_value_initguess_prepare}
    A_1+A_2-A_3\leq V_T(\tint,\yint)-  V^r_T(\tint, \tildeyint)\leq  A_1+A_2+A_3,
\end{align}
where $A_1$, $A_2$, and $A_3$ are defined as 
\begin{align*}
    A_1 \coloneqq &  \tfrac{1}{2}\inf\big\{L'_y(\bar x^r; \tildeyint)(\bar y-y^r)+L_p'(\bar x^r;\tildeyint)(\bar p-p^r)\,\big|\,y^r,p^r\in \mY^r_T(\tint)\big\},\\
    A_2 \coloneqq & {\langle {\tildeyint}-\yint, \bar { p}^r(\tint)\rangle}_H + \tfrac{1}{2} {\langle \yint-{ \tildeyint}, (\bar p -\bar { p}^r)(\tint)\rangle}_H,\\
    A_3 \coloneqq & \tfrac{1}{2}\big(L'_{\bu}(\bar x^r; \tildeyint)-L'_{\bu}(\bar x; \yint)\big)(\bar \bu -\bar \bu ^r).
\end{align*}

\end{theorem}
\begin{proof}
   By definition
    \begin{align*}
         L(\bar x;\yint)+g_T(\bar \bu)-L(\bar { x}^r; \tildeyint)-g_T(\bar \bu^r) &= J_T(\bar \bu;\tint,\yint)-J^r_T(\bar { \bu}^r;\tint,\tildeyint)\\
         &= V_T(\tint,\yint)- V^r_T(\tint,\tildeyint),
    \end{align*}
    and by adding $\pm {\langle \yint-{\reduce {\bar y}}(\tint), \bar { p}^r(\tint)\rangle}_H$, we get
    \begin{align}
         L(\bar x;\yint)-L(\bar { x}^r;\tildeyint)  
         & =  L(\bar x;\yint) - L(\bar { x}^r; \yint) + {\langle \yint-{\tildeyint}, \bar { p}^r(\tint)\rangle}_H.\label{eq:valueapost_help1}
    \end{align}
    For $e=(e_y,e_\bu,e_p)=\bar x-\bar { x}^r\in\mX_T(\tint)$, we use the fundamental theorem of calculus to get
    \begin{align}
        \begin{aligned}
            L(\bar x;\yint)-L(\bar { x}^r;\tildeyint)&\stackrel{\eqref{eq:valueapost_help1}}{=}L(\bar x^r+e;\yint)-L(\bar { x}^r;\yint)+{\langle \yint-{\tildeyint},\bar { p}^r(\tint)\rangle}_H  \\
            & = \int_{0}^1 L'(\bar { x}^r+t e;\yint)(e)\,\dt +{\langle \yint-{\tildeyint}, \bar { p}^r(\tint)\rangle}_H.\\
        \end{aligned}
        \label{eq:valueapost_helper2}
    \end{align}
    For the part depending on $g_T$, we obtain from the FOM optimality conditions \eqref{eq:FOML3} for $\bu = \bar \bu^r$
    \begin{align}
        g_T(\bar \bu) -g_T(\bar \bu^r)\leq L'_{\bu}(\bar x;\yint)(-e_\bu).
        \label{eq:valueapost_helper3}
    \end{align}
    Combining \eqref{eq:valueapost_helper2}, \eqref{eq:valueapost_helper3} and the \FOM\  optimality conditions \eqref{eq:FOML1} for $y=e_y$, \eqref{eq:FOML2} for $p=e_p$, leads to
    \begin{align*}
        &L(\bar x;\yint)+g_T(\bar \bu)-L(\bar x^r;\tildeyint)-g_T(\bar \bu^r)\\
        &\leq \int_0^1 L'(\bar x^r+t e;\yint)(e)\,\dt+{\langle \yint- {\tildeyint}, \bar p^r(\tint)\rangle}_H\pm \tfrac{1}{2}L'(\bar x^r;\yint)(e)\\
         &\quad+L'_{\bu}(\bar x;\yint)(-e_\bu)+\tfrac{1}{2}L'_y(\bar x;\yint)(-e_y)+\tfrac{1}{2}L'_p(\bar x;\yint)(-e_p)\\
        &=\int_0^1 L'(\bar x^r+t e;\yint)(e)\,\dt+{\langle \yint-{\tildeyint}, \bar p^r(\tint)\rangle}_H\pm \tfrac{1}{2}L'(\bar x^r;\yint)(e)\\
         &\quad-\tfrac{1}{2} L'(\bar x;\yint)(e)+\tfrac{1}{2}L'_{\bu}(\bar x;\yint)(-e_\bu).
    \end{align*}
    Due to the quadratic structure of the Lagrange functional $L$, the integral and the two terms $-\tfrac{1}{2}L'(\bar x;\yint)(e)$ and $-\tfrac{1}{2}L'(\bar x^r;\yint)(e)$ cancel (cf. \cite{ran}), and we obtain
    \begin{align*}
       V_T(\tint,\yint)- V^r_T(\tint,\tildeyint)
        &\leq {\langle {\tildeyint}-\yint,\bar p^r(\tint)\rangle}_H+\tfrac{1}{2}L'(\bar x^r;\yint)(e)+\tfrac{1}{2}L'_{\bu}(\bar x; \yint)(-e_\bu).
    \end{align*}
    In addition, we have
    \begin{align} \label{eq:valueapost_helper4}
        \tfrac{1}{2}L'(\bar { x}^r;\yint)(e) = \tfrac{1}{2}L'(\bar { x}^r; \tildeyint)(e) + \tfrac{1}{2}{\langle \yint- { \tildeyint},e_p(\tint)\rangle}_H,
    \end{align}
    which leads to
    \begin{align*}
        &V_T(\tint,\yint)- V^r_T(\tint,\tildeyint)\\
        &\leq{\langle \yint-{\tildeyint}, \bar { p}^r(\tint)\rangle}_H + \tfrac{1}{2} {\langle \yint- {\tildeyint},e_p(\tint)\rangle}_H
        + \tfrac{1}{2}L'(\bar x^r; \tildeyint)(e) + \tfrac{1}{2}L'_{\bu}(\bar x; \yint)(-e_\bu)\\
        &=A_2
        + \tfrac{1}{2}\big(L'_{\bu}(\bar x^r; \tildeyint)-L'_{\bu}(\bar x; \yint)\big)(e_u)
+ \tfrac{1}{2}L'_y(\bar x^r; \tildeyint)(e_y)+\tfrac{1}{2}L'_p(\bar x^r; \tildeyint)(e_p)\\
        &= A_2+ A_3+\tfrac{1}{2}L'_y(\bar x^r; \tildeyint)(e_y)+\tfrac{1}{2}L'_p(\bar x^r; \tildeyint)(e_p).
    \end{align*}
    Further, using \eqref{eq:ROML1} and \eqref{eq:ROML2}, we obtain 
    \begin{align}
        \label{eq:valueapost_helper5}
        \begin{aligned}
            \tfrac{1}{2}L_y'(\bar x^r; \tildeyint)(\bar y-\bar y^r)&= \tfrac{1}{2}L_y'(\bar x^r; \tildeyint)(\bar y- y^r),\\
            \tfrac{1}{2}L_p'(\bar x^r; \tildeyint)(\bar p-\bar p^r)&= \tfrac{1}{2}L_p'(\bar x^r; \tildeyint)(\bar p- p^r)
        \end{aligned}
    \end{align}
    for arbitrary $ y^r, p^r\in \mY^r_T(\tint)$. This implies the structure of $A_1$ and therefore the upper bound in \eqref{eq:errorest_value_initguess_prepare}.
    For the lower bound, we consider similar arguments as for the upper bound using \eqref{eq:ROML3} for $\bu=\bar \bu$, \eqref{eq:valueapost_help1} and \eqref{eq:valueapost_helper4}, to estimate
    \begin{align*}
        &L(\bar x;\yint)+g_T(\bar \bu)-L(\bar x^r;\tildeyint)-g_T(\bar \bu^r)\\
        &\quad\geq\int_0^1 L'(\bar x^r+t e;\yint)(e)\,\dt+{\langle \yint-{\tildeyint}, \bar p^r(\tint)\rangle}_H\pm \tfrac{1}{2}L'(\bar x^r;\yint)(e)\\
        &\qquad-L'_\bu(\bar x^r;\tildeyint)(e_\bu)+\tfrac{1}{2}L'_y(\bar x;\yint)(-e_y)+\tfrac{1}{2}L'_p(\bar x;\yint)(-e_p)\\
        &\quad= {\langle \yint-{\tildeyint}, \bar p^r(\tint)\rangle}_H+ \tfrac{1}{2} {\langle \yint- {\tildeyint},e_p(\tint)\rangle}_H+\tfrac{1}{2}L'(\bar x^r;\tildeyint)(e)\\
        &\qquad - L_\bu'(\bar x^r;\tildeyint)(e_\bu)+\tfrac{1}{2}L_\bu'(\bar x;\yint)(e_\bu)\\
        &\quad= A_2-A_3+ \tfrac{1}{2}L'_y(\bar x^r; \tildeyint)(e_y)+\tfrac{1}{2}L'_p(\bar x^r; \tildeyint)(e_p).
    \end{align*}
    Now \eqref{eq:valueapost_helper5} implies the lower bound in \eqref{eq:errorest_value_initguess_prepare}.
\end{proof}
To obtain a computable bound, we invoke the error estimates from the last subsection.
\begin{corollary}[Value function error estimate $\Delta_{V_T}$]
    \label{cor:errorest_optimalval_diff_initguess}
    In the situation of \cref{lem:errorest_optimalval_diff_initguess_prepare}, we have the bound
    \begin{equation}
        \label{eq:errorest_value_initguess}
        |V_T(\tint,\yint)-  V^r_T(\tint, \tildeyint)|\leq \Delta_{V_T}(\tint, \tildeyint, \Delta_{\yint})
    \end{equation}
    with
    \begin{align}
        \label{eq:valueest_line1}
        \begin{aligned}
            \Delta_{V_T}(\tint, \tildeyint, \Delta_{\yint})&\coloneqq\tfrac{1}{2}\sqrt{\norm{\Res_p(\bar p^r,\bar y^r)}{L^2(V')}^2}\bar \Delta_{y}\\
            &\quad+\tfrac{1}{2}\sqrt{\norm{\Res_y(\bar y^r,\bar \bu^r)}{L^2(V')}^2+\norm{\bar y^r(\tint)-\tildeyint}{H}^2}\bar \Delta_{p}\\
            &\quad+\tfrac{1}{2}\norm{B}{L^\infty} \bar\Delta_{p}\bar \Delta_{\bu}+\Delta_{\yint}\norm{\bar { p}^r(\tint)}{H}+ \tfrac{1}{4}\Delta_{\yint}\bar \Delta_{p},
        \end{aligned}
    \end{align}
    and $ \bar\Delta_y= \bar\Delta_y(\tint,\tildeyint)$, $ \bar\Delta_p= \bar\Delta_p(\tint,\tildeyint)$ defined as
    \begin{equation}\label{eq:optimal_error_ests1}
        \bar \Delta_y \coloneqq \Delta_{y}(\bar \Delta_\bu, \Delta_{\yint}, \bar y^r, \bar \bu^r), \quad  \bar \Delta_p\coloneqq \Delta_{p}(\bar\Delta_{y,H}, \bar p^r, \bar y^r)
    \end{equation}
    for $\Delta_{\yint},\Delta_{y},\Delta_p$, $\bar \Delta_\bu$ and $\bar\Delta_{y,H}$ as in \eqref{eq:perturb_init}, \cref{Lemma:ResBasStateError}, \cref{lem:ResAdjest} and \cref{theo:oc_est}, respectively.
\end{corollary}

\begin{proof}
    From \eqref{eq:errorest_value_initguess_prepare}, it follows
    \begin{equation*}
        |V_T(\tint,\yint)-  V^r_T(\tint, \tildeyint)|\leq \max\{A_1+A_2+A_3,-A_1-A_2+A_3\}\leq |A_1|+|A_2|+A_3.
    \end{equation*}
    By \cref{theo:oc_est}, we can bound the error in the optimal control by $\bar\Delta_\bu$. Since we have $\bar y = y(\bar \bu,\yint)$ and $\bar y^r=y^r(\bar \bu^r, \tildeyint)$, we can estimate the optimal state error by \cref{Lemma:ResBasStateError} using $\bar \Delta_{y}$ in \eqref{eq:optimal_error_ests1} as
    \begin{align}
        \label{eq:help22}
        \norm{(\bar y -\bar y^r)(\tint+T)}{H}^2+ \norm{\bar y -\bar y^r}{L^2(\tint,\tint+T;V)}^2\leq \bar \Delta_y^2.
    \end{align}
    For the adjoint state $\bar p = p(\bar y)$ and $\bar p^r=p^r(\bar y^r)$, we obtain by \cref{lem:ResAdjest} for $\Delta_y=\bar \Delta_{y,H}$ in \eqref{eq:opt_out_est} and the definition of $\bar \Delta_{ p}$ in \eqref{eq:optimal_error_ests1}
    \begin{align}
        \label{eq:cor:errorest_optimalval_diff_initguess_help1}
        \norm{(\bar p -\bar p^r)(\tint)}{H}^2+ \norm{\bar p -\bar p^r}{L^2(\tint,\tint+T;V)}^2\leq \bar \Delta^2_p.
    \end{align}
    For the error in the optimal gradient, we obtain using the structure of $L'_\bu(x^r;\yint)$ ($L'_\bu(x,\tildeyint)$) in \eqref{eq:FOM_opcond} (\eqref{eq:ROM_opcond}), the optimal control estimate \eqref{eq:opt_cont_est} and \eqref{eq:cor:errorest_optimalval_diff_initguess_help1}
    \begin{align}
       \nonumber
            A_3&= \tfrac{1}{2}\big(L'_{\bu}(\bar x^r; \tildeyint)-L'_{\bu}(\bar x; \yint)\big)(e_\bu)\stackrel{\eqref{eq:FOM_opcond},\eqref{eq:ROM_opcond}}{=}\tfrac{1}{2}\langle B'(\bar p^r-\bar p)+\lambda(\bar \bu^r-\bar \bu) , \bar \bu-\bar \bu^r\rangle_{\mU_T(\tint)}\\
            &\leq \tfrac{1}{2} \langle B'(\bar p^r-\bar p), \bar \bu-\bar \bu^r\rangle_{\mU_T(\tint)}-\lambda\norm{\bar \bu^r-\bar \bu}{\mU_T(\tint)}^2\stackrel{\eqref{eq:opt_cont_est}, \eqref{eq:cor:errorest_optimalval_diff_initguess_help1}}{\leq} \tfrac{1}{2}\norm{B}{L^\infty} \bar \Delta_{p}\bar \Delta_\bu.   \label{eq:helper} 
    \end{align}
    Choosing $y^r=\bar y^r $, $p^r= \bar p^r$ in the term $A1$ in \eqref{eq:errorest_value_initguess_prepare}, and upper bounding the second line of \eqref{eq:errorest_value_initguess_prepare} using \eqref{eq:helper}, \eqref{eq:perturb_init} leads to 
    \begin{align*}
        |V_T(\tint,\yint)-  V^r_T(\tint, \tildeyint)|\leq &   
       \tfrac{1}{2}|L'_y(\bar x^r; \tildeyint)(\bar y-\bar y^r)|+\tfrac{1}{2}|L_p'(\bar x^r;\tildeyint)(\bar p-\bar p^r)| \\
        & +\tfrac{1}{2}\norm{B}{L^2} \bar\Delta_{p}\bar\Delta_{\bu}  +\Delta_{\yint}{\norm{\bar { p}^r(\tint) }{H}}+ \tfrac{1}{4}\Delta_{\yint}\bar \Delta_p.
    \end{align*}
    For the last term, we have used that the factor of $\bar \Delta_p$ can be improved by $\nicefrac{1}{2}$ according to \eqref{eq:adjstate_apost_Ht}. Now we estimate the remaining derivatives of $L$. Note that $L_p'(\bar x^r;\tildeyint)\in \mY_T(\tint)'$ is continuous in $Z\coloneqq L^2(\tint,\tint+T;V)\cap C([\tint,\tint+T];H)\supset \mY_T(\tint)$ with norm $\norm{v}{Z}^2\coloneqq \norm{v(\tint)}{H}^2 +\norm{v}{L^2(V)}^2$, since for all $v\in Z$
    \begin{align*}
        |L_p'(\bar x^r;\tildeyint)(v)|\leq &|\langle \Res_y(\bar y^r,\bar \bu^r),v\rangle_{L^2(\tint,\tint+T;V)}|+|\langle \bar y^r(\tint)-\tildeyint, v(\tint)\rangle_H|\\
        \leq & {\sqrt{\norm{\Res_y(\bar y^r,\bar \bu^r)}{L^2(\tint,\tint+T;V')}^2+\norm{\bar y^r(\tint)-\tildeyint}{H}^2}}\norm{v}{Z}.
    \end{align*}
    {Setting $v= \bar p-\bar p^r$}, we conclude by using the error estimate for the adjoint
    \begin{align*}
        |L_p'(\bar x^r;\tildeyint)(\bar p-\bar p^r)|&\leq \norm{L_p'(\bar x^r;\tildeyint)}{Z'}\norm{\bar p-\bar p^r}{Z}\\
        &\hspace{-2mm}\stackrel{\eqref{eq:cor:errorest_optimalval_diff_initguess_help1}}{\leq} \sqrt{\norm{\Res_y(\bar y^r,\bar \bu^r)}{L^2(\tint,\tint+T;V')}^2+\norm{\bar y^r(\tint)-\tildeyint}{H}^2 }\bar \Delta_{p}.
    \end{align*}
    For $|L'_y(\bar x^r; \tildeyint)(\bar y-\bar y^r)|$, similar arguments using the norm $\norm{v}{Z}^2\coloneqq \norm{v(\tint+T)}{H}^2 +\norm{v}{L^2(V)}^2$, and \eqref{eq:help22} leads to the result.
\end{proof}
The next result compares the FOM and ROM cost at a fixed control $\bu\in\mU_T(\tint)$.
\begin{theorem}[Cost function error representation]
    \label{lem:errorest_optimalval_not_optimal}
    Suppose that \cref{ass:pde,ass:running_cost} hold and let $\bu\in\mU_T(\tint)$. For $x(\bu)\coloneqq(y(\bu),\bu,p(y(\bu)))\in\mX_T(\tint)$ and $x^r(\bu)\coloneqq(y^r(\bu),\bu,p^r(y^r(\bu)))\in\mX^r_T(\tint)$, we have the following error representation
    \begin{align}\nonumber
        &J_T(\bu ;\tint,\yint)- J^r_T(\bu;\tint,\tildeyint)\\
        &\label{eq:infCostest}\quad=\tfrac{1}{2}\inf\big\{L_y'(x^r(\bu);\tildeyint)(y(\bu)-y^r)+L_p'(x^r(\bu);\tildeyint)(p(y(\bu))-p^r)\,\big|\,y^r,p^r\in \mY^r_T(\tint)\big\}\\
        &\qquad+ {\langle {\tildeyint}-\yint, p^r(y^r(\bu))(\tint)\rangle}_H + \tfrac{1}{2} {\langle \yint-{ \tildeyint}, (p(y(\bu))-p^r(y^r(\bu)))(\tint)\rangle}_H.\nonumber
    \end{align}
\end{theorem}

\begin{proof}
    It holds that $
    J_T(\bu ;\tint,\yint)- J^r_T(\bu;\tint,\tildeyint) = L(x(\bu);\yint)-L(x^r(\bu) ;\tildeyint)$ 
    and $L_{(y,p)}'(x(\bu);\yint)(y,p)=0$ for all $(y,p)\in \mY_T(\tint)\times\mY_T(\tint)$ and $L_{(y,p)}'(x^r(\bu);\tildeyint)(y^r,p^r)=0$ for all $(y^r,p^r)\in\mY_T^r(\tint)\times\mY_T^r(\tint)$. 
    Applying analogous argumentations as in the proof of \cref{lem:errorest_optimalval_diff_initguess_prepare} to $e=x(\bu)-x^r(\bu)= (e_y,0,e_p)$ leads to the desired result.
\end{proof}

\begin{corollary}[Cost function error estimate $\Delta_{J_T}$]
    \label{cor:costbound}
    In the situation of \cref{lem:errorest_optimalval_not_optimal} we have
    \begin{equation}
        \label{eq:errorest_value_not_optimal}
        |J_T(\bu ;\tint,\yint)- J^r_T(\bu;\tint,\yint)|\leq \Delta_{J_T}(\tint,\tildeyint,\bu,\Delta_{\yint})
    \end{equation}
    with 
    \begin{align*}
       \Delta_{J_T}(\tint, \tildeyint, \bu, \Delta_{\yint} )&\coloneqq \tfrac{1}{2}\sqrt{\norm{\Res_p(p^r(\bu), y^r(\bu))}{L^2(V')}^2}\Delta_{y}(\bu)\\
        &\quad+\tfrac{1}{2}\sqrt{\norm{\Res_y(y^r(\bu),\bu)}{L^2(V')}^2+\norm{y^r(\tint;\bu)-\tildeyint}{H}^2}\Delta_{p}(\bu)\\
        &\quad+\Delta_{\yint}{\norm{\bar { p}^r(\tint) }{H}}+ \tfrac{1}{4}\Delta_{\yint}\Delta_{p}(\bu)
    \end{align*}
    and 
    \begin{equation*}
        \Delta_y(\bu) \coloneqq \Delta_{y}(0, \Delta_{\yint}, y^r(\bu), \bu),\quad  \Delta_p(\bu)\coloneqq \Delta_{p}(\Delta_{y}(\bu), p^r(y^r(\bu)), y^r(\bu))
    \end{equation*}
    for $\Delta_{y},\Delta_p$ as in \cref{Lemma:ResBasStateError} and \cref{lem:ResAdjest} respectively.
\end{corollary}

\begin{proof}
   Choosing $y^r= y^r(\bu)$, $p^r= p^r(y^r(\bu))$ in \eqref{eq:infCostest}, and upper bounding the remaining terms similarly as in \cref{cor:errorest_optimalval_diff_initguess} leads to the result.
\end{proof}
\begin{remark}

{ For the evaluation of the value function a posteriori error estimator from
\cref{cor:errorest_optimalval_diff_initguess}, one needs to compute the reduced optimal
triple \(\bar x^r = (\bar y^r, \bar{\bu}^r, \bar p^r)\) by solving \eqref{eq:ROML}, together
with the \(L^2(V')\)-norms of the residuals \(\calR_p(\bar p^r, \bar y^r)\) and
\(\calR_y(\bar y^r, \bar{\bu}^r)\).
For an online-efficient evaluation, an offline--online decomposition can be employed
(cf.~\cite{hesthaven2016certified}; see also the discussion in
\cref{subsec:NUMEXP_setup}). Similarly, the cost functional error estimator from \cref{cor:costbound} requires the
evaluation of the reduced state \(y^r(\bu)\) and the corresponding adjoint
\(p^r(y^r(\bu))\), together with the associated residual norms.}
{ Note that, if \(\Delta_{\yint}=0\), both the value function and cost functional
estimators scale quadratically with respect to the residuals. Indeed, for instance, in
\eqref{eq:valueest_line1}, the state residual \(\calR_y\) is multiplied by the error
estimator for the optimal adjoint state \(\bar{\Delta}_p\), which depends linearly on the
adjoint residual (see \eqref{eq:adjstate_apost_CHL2V_struc}).}
\end{remark}
\subsection{Properties of the ROM}\label{subsec:ROMproperties}

In this section, we state interpolation properties of the ROM and its error estimators and provide convergence results.
\begin{lemma}[Properties of the ROM]\label{lem:interpolation_properties}
    Let \cref{ass:pde,ass:running_cost} be valid and assume $\Delta_{\yint}=0$ in \eqref{eq:perturb_init}, i.e., $\yint=\tildeyint\in H$. It holds
    \begin{enumerate}
        \item\label{lem:interpolation_properties_1} Interpolation property: $y(\bu),\ p(y(\bu))\in \mY_T^r(\tint) \Rightarrow y(\bu)=y^r(\bu)\text{ and }p(y(\bu))=p^r(y^r(\bu))$;
        \item\label{lem:interpolation_properties_2} Interpolation of the optimality system 1: $\bar y,\bar p \in \mY_T^r(\tint) \Rightarrow \bar x=\bar x^r$;
        \item\label{lem:interpolation_properties_3} Interpolation of the optimality system 2: $ y(\bar \bu^r), p(y(\bar \bu^r)) \in \mY_T^r(\tint) \Rightarrow \bar x=\bar x^r$;
        \item\label{lem:interpolation_properties_4} $\bar x=\bar x^r\Rightarrow \Delta_{V_T}=0$;
        %
        %
    \end{enumerate}
\end{lemma}

\begin{proof}
    \begin{enumerate}
        \item Given that \( y(\bu) \) and \( p(y(\bu)) \) belong to \( \mY_T^r(\tint) \), they solve the ROM state and adjoint equations, respectively. Hence, the assertion follows from the uniqueness of these solutions.
        \item Due $\bar y,\bar p \in \mY_T^r(\tint)$ and \cref{lem:interpolation_properties_1}, we have $\bar y = y^r(\bar \bu)$ and therefore $\bar p = p^r(y^r(\bar \bu))$. Hence, $\bar \bu$ fulfills the optimality condition of the ROM \eqref{eq:ROM_opcond} and it follows $\bar \bu = \bar \bu^r$ by uniqueness. 
        \item Due to $ y(\bar \bu^r), p(y(\bar \bu^r)) \in \mY_T^r(\tint)$ and \cref{lem:interpolation_properties_1}, we have $y(\bar \bu^r) = y^r(\bar \bu^r)$ and therefore $p(y(\bar \bu^r)) = p^r(y^r(\bar \bu^r))$. Hence, $\bar \bu^r$ fulfills the optimality condition of the FOM \eqref{eq:FOM_opcond} and it follows $\bar \bu = \bar \bu^r$ by uniqueness. 
        \item Due to $\bar x = \bar x^r$ and \cref{lemma:app_error_equivalence}, the two error estimators $\Delta_y(0,0,\bar \bu^r, \bar y^r)$ and $\Delta_p(0,\bar \bu^r,\bar y^r)$ vanish. According to \cref{theo:oc_est} the control estimator $\bar \Delta_u$ vanishes and hence $\Delta_{V_T}=0$ by definition in \eqref{eq:valueest_line1}.
        %
        %
    \end{enumerate}
\end{proof}
In the following, we show the convergence of $\bar x^r$ to $\bar x$. The proof and some preliminary lemmas are deferred to \cref{sec:APP_add_estimates}.
\begin{theorem}\label{theo:os_conv}
    Let $\bar x=(\bar y, \bar \bu, \bar p)\in \mX$ and $\bar x^r=(\bar y^r, \bar \bu^r, \bar p^r)\in \mX^r$ be the solution of \eqref{eq:FOML} and \eqref{eq:ROML}, respectively, for $\tildeyint=\yint\in H$ and $V_r\coloneqq \linspan\{v_1,\ldots,v_r\}$ and an orthonormal basis $(v_n)_{n\in \N}$ of $V$. Then, $\bar x^r\to \bar x$ in $\mX(\tint)$ as $r\to \infty$.
\end{theorem}
%
\section{Suboptimality of the reduced RHC algorithm}\label{sec:stabROMRHC}
%
Replacing the FOM finite-horizon problem \eqref{eqn:fin_hor_ocp} with its reduced counterpart \eqref{eqn:reducedStabilizationProblem} to determine a control for the full-order system \eqref{eq:LTV} may compromise the exponential stability of the resulting closed-loop system. In this section, we investigate conditions, formulated in terms of the error estimates, under which the reduced optimal control stabilizes the full-order system with a desired performance level \( \tilde{\alpha} \in (0,1) \). A central component of our analysis is the following formulation of the RDP.
\begin{theorem}[Relaxed Dynamic Programming]\label{theorem:RelaxedDDP}
    Let \cref{ass:pde,ass:running_cost} hold and let $T\geq \delta >0$. Let $\bu\in \mU_\infty(0)$ be a control with trajectory $y=y(\bu,0,y_0)$ starting from $y_0\in H$. If there exists $\tilde\alpha \in (0,1)$ such that
    \begin{align}\label{eq:RelaxedDDP_dissipativity}
        \tilde\alpha J_\delta(\bu;t_k,y(t_k))&\leq  V_T(t_k,y(t_k)) - V_T(t_{k+1},y(t_{k+1}))  \quad \text{ for } k \in \N_0,
    \end{align}
    then $y\in L^2(0,\infty;V)$ and the following \emph{suboptimality inequality} holds
    \begin{equation}
        \label{eq:RelaxedDDP_suboptimaity_inequality}
        V_\infty(y_0)\leq J_\infty(\bu;0,y_0)\leq \tfrac{1}{\tilde\alpha}V_T(0,y_0)\leq \tfrac{1}{\tilde\alpha}V_\infty(0,y_0).
    \end{equation}
    If additionally \cref{ass:1} holds, we have
    \begin{align}\label{eq:relDD_help1}
        V_T(t_{k+1},y(t_{k+1})) \leq e^{-\tilde\zeta t_k}V_T(0,y_0),\\
         \label{eq:RelaxedDDP_exponential_stability}
        {\norm{y(t;\bu,y_0)}{H}^2} \leq \tilde C_{rh} e^{-\tilde\zeta t}\norm{y_0}{H}^2 \quad  \text{ for }t\geq 0,
    \end{align}
    where the positive numbers $(\tilde\zeta,\tilde C_{rh})$ depend on $(\tilde\alpha,\delta,T)$. 
\end{theorem}
\begin{proof}
    Inequality \eqref{eq:RelaxedDDP_suboptimaity_inequality} follows directly from \eqref{eq:RelaxedDDP_dissipativity} with the same argument as in the proof of \cref{theorem:stabilityRHC}. Next, we turn to \eqref{eq:relDD_help1}, focusing first on the discrete time steps $t_k$. From \cref{ass:1} and \cref{ass:2}, we obtain
    \begin{align*}
         V_T(t_k,y(t_k)) - V_T(t_{k+1},y(t_{k+1})) & \stackrel{\eqref{eq:RelaxedDDP_dissipativity}}{\geq} \tilde\alpha J_\delta(\bu;t_k,y(t_k))\stackrel{\eqref{e8a}}{\geq} \tfrac{\tilde\alpha}{C_2(\delta)}\norm{y(t_{k+1})}{H}^2\\
         &\stackrel{\eqref{e7}}{\geq}\tfrac{\tilde\alpha}{C_2(\delta)\gamma(T))}V_T(t_{k+1},y(t_{k+1})).
    \end{align*}
    Hence, with $\tilde \eta=(1+\nicefrac{\tilde\alpha}{C_2(\delta)\gamma(T)})^{-1}\in (0,1)$ it holds $V_T(t_{k+1},y(t_{k+1})) \leq\tilde \eta V_T(t_k,y(t_k))$. Setting $\tilde\zeta=\nicefrac{|\ln(\tilde\eta)|}{\delta}$, this inductively implies \eqref{eq:relDD_help1}. Turning to \eqref{eq:RelaxedDDP_exponential_stability}, we have
    \begin{equation}
    \label{e132}
    \begin{split}
    \norm{y(t_{k+1})}{H}^2 &\stackrel{\text{\eqref{e8a}}}{\leq} C_2(\delta) J_\delta(\bu,t_k,y(t_k)) \stackrel{\text{\eqref{eq:RelaxedDDP_dissipativity}}}{\leq} \tfrac{C_2(\delta)}{\tilde \alpha} V_T( t_{k} , y(t_{k}))\\
    & \stackrel{\text{\eqref{eq:relDD_help1}}}{\leq} \tfrac{C_2(\delta)}{\tilde \alpha}  e^{-\tilde\zeta t_{k}} V_T(0, y_0) \stackrel{\text{\eqref{e7}}}{\leq} \tfrac{C_2(\delta)\gamma(T)}{\tilde \alpha}  e^{-\tilde\zeta t_{k}} \norm{y_0}{H}^2.
    \end{split}
    \end{equation}
    Setting $\tilde C_H=\tfrac{ C_2(\delta) \gamma(T)}{\tilde\alpha\tilde\eta}$, we get $\norm{y(t_{k+1})}{H}^2 \leq \tilde C_H e^{-\tilde\zeta t_{k+1}} \norm{y_0}{H}^2$ for $k \in \N_0$.
    We obtain then \eqref{eq:RelaxedDDP_exponential_stability} with $\tilde C_{rh} \coloneqq\bar{C}\tilde C_H(1+\gamma(T))\tilde\eta^{-1}$, where we use the same arguments as in \eqref{eq:exp_cont_in_t} replacing all the constants there with their corresponding ones here.
\end{proof} 
\begin{definition}\label{def:current_perfind}
    Given a control $\bu \in \mU_\infty(0)$, current initial value $\yint\in H$ at $\tint \geq 0$, we define the induced perfomance index as
    \begin{equation}\label{eq:current_pfind}
        \alpha(\bu,\tint,\yint)\coloneqq \frac{ V_T(\tint, \yint) -  V_T(\tint+\delta, y(\tint+\delta;\bu,\tint,\yint)}{J_\delta(\bu;\tint,\yint) }.
    \end{equation}
    The above definition is equally applicable for $\bu \in \mU_\delta(\tint)$.
\end{definition}
We adaptively construct the reduced space $V_r$ and the reduced RHC $\bu_{rh}^r\in \mU_\infty(0)$ to satisfy condition \eqref{eq:RelaxedDDP_dissipativity} from \cref{theorem:RelaxedDDP}, and denote the corresponding trajectory by $\tilde y_{rh}=y(\bu_{rh}^r,0,y_0)$. Let $\tilde y_k=\tilde y_{rh}(t_k)\in H$ for $k\in \N_0$ and consider the reduced optimal control $\bar \bu_k^r \coloneqq \bar \bu^r(t_k,\tilde y_k)\in \mU_T(t_k)$ being the solution to \eqref{eqn:reducedStabilizationProblem} for $(\tint, \tildeyint)=(t_k, \tilde y_k)$. With \cref{def:current_perfind}, condition \eqref{eq:RelaxedDDP_dissipativity} reads as
\begin{equation}
    \label{eq:ROMsuboptimality}
     \alpha_k \coloneqq \alpha(\bar\bu^r_k,t_k,\tilde y_k)\geq {\tilde \alpha} \quad \text{for all }k\in \N_0.
\end{equation}
Unfortunately, verifying \eqref{eq:ROMsuboptimality} is computationally expensive, as it requires two evaluations of the full-order value function. Therefore, we instead consider inexpensive sufficient and necessary conditions based on the ROM value function $V_T$ and the error estimator $\Delta_{V_T}$ introduced in \cref{cor:errorest_optimalval_diff_initguess}. A sufficient condition for \eqref{eq:ROMsuboptimality} provided in \cite{dietze2023reduced} is
\begin{align}
    \label{eq:ROMsuboptimality_suff}
    \begin{aligned}
        \ubar{\alpha}_k&\coloneqq\tfrac{ V_T^r(t_k, \tilde y_k)-\Delta_{V_T}(t_k,\tilde y_k,0)-V_T^r(t_{k+1},\tilde y_{k+1})-\Delta_{V_T}(t_{k+1},\tilde y_{k+1},0)}{J_\delta(\bar \bu_k^r;t_k,\tilde y_k)}\geq {\tilde \alpha},
    \end{aligned}
\end{align}
while a necessary condition is 
\begin{align}
    \label{eq:ROMsuboptimality_necc}
    \begin{aligned}
        \tilde \alpha \leq \bar{\alpha}_k&\coloneqq\min\Big\{\tfrac{V_T^r(t_k, \tilde y_k)+\Delta_{V_T}(t_k,\tilde y_k,0)-V_T^r(t_{k+1},\tilde y_{k+1})+\Delta_{V_T}(t_{k+1},\tilde y_{k+1},0)}{J_\delta(\bar \bu_k^r;t_k,\tilde y_k))},1\Big\}
    \end{aligned}
\end{align}
since by the error estimator property \eqref{eq:errorest_value_initguess} it holds $\alpha_k\in [\ubar{\alpha}_k, \bar{\alpha}_k]$. 
A disadvantage of the pair $(\ubar{\alpha}_k, \bar{\alpha}_k)$ is that their computation at time step~$k$ still requires computing the full-order system over the interval $[t_k, t_k + \delta]$ -- see, for instance, the dependency on the FOM state $\tilde y_{k+1}$ and the FOM cost $J_\delta$ in the denominator \eqref{eq:ROMsuboptimality_suff}. {To make the certification procedure completely independent from this FOM evaluation}, we introduce a new performance index that enables an online-efficient acceptance procedure at step~$k$, {at the cost of a larger overestimation (see the discussion in
\cref{subsec:numstab})}. Specifically, we replace the FOM state $\tilde y_{k+1} = y(t_{k+1};\bar \bu_k^r, t_k, \tilde y_k)$ and the cost $J_\delta(\bar \bu_k^r; t_k, \tilde y_k)$ with their ROM counterparts $y^r_{k+1} = y^r(t_{k+1}; \bar \bu_k^r, t_k, \tilde y_k)$,  $J^r_\delta(\bar \bu_k^r; t_k, \tilde y_k)$, and incorporate the estimators $\Delta_{\tilde y_{k+1}} \coloneqq \Delta_{y}(t_{k+1}; 0, 0, y^r(\bar \bu_k^r, t_k, \tilde y_k), \bar \bu_k^r),\ \Delta_{J_\delta}(\bar \bu_k^r, t_k, \tilde y_k, 0)$ (see \eqref{eq:state_apost_CHL2V} for $t=t_{k+1}$ in \cref{Lemma:ResBasStateError} and \cref{cor:costbound}). We define
    \begin{align}
        \label{eq:ROMsuboptimality_suff_fullROM}
        \begin{aligned}
            \ubar{\alpha}_k^r&\coloneqq \tfrac{V_T^r(t_k, \tilde y_k)-\Delta_{V_T}(t_k, \tilde y_k,0)-V_T^r(t_{k+1},y^r_{k+1})-\Delta_{V_T}(t_{k+1},y^r_{k+1},\Delta_{ \tilde y_{k+1}})}{J^r_\delta(\bar \bu_k^r;t_k, \tilde y_k)+\Delta_{J_\delta}(t_k,  \tilde y_k, \bar \bu_k^r, 0 )}
        \end{aligned}
    \end{align}
    and
    \begin{align}
        \label{eq:ROMsuboptimality_necc_fullROM}
        \begin{aligned}
            \bar{\alpha}_k^r&\coloneqq\min \Big\{\tfrac{V_T^r(t_k,  \tilde y_k)+\Delta_{V_T}(t_k, \tilde y_k,0)-V_T^r(t_{k+1},y^r_{k+1})+\Delta_{V_T}(t_{k+1},y^r_{k+1},\Delta_{ \tilde y_{k+1}})}{J^r_\delta(\bar \bu_k^r;t_k, \tilde y_k)-\Delta_{J_\delta}(t_k,  \tilde y_k, \bar \bu_k^r, 0 )},1\Big\}.
        \end{aligned}
    \end{align}
    Then $\alpha_k\in [\ubar{\alpha}^r_k, \bar{\alpha}^r_k]$ by \cref{cor:errorest_optimalval_diff_initguess,cor:costbound} and the condition $\ubar{\alpha}^r_k \geq \tilde{\alpha}$ is sufficient to guarantee \eqref{eq:ROMsuboptimality}. Note that for the computation of $(\ubar{\alpha}_k, \bar{\alpha}_k)$, $(\ubar{\alpha}_k^r, \bar{\alpha}_k^r)$, we need to evaluate the reduced value function at step $k$ and $k+1$.
In the following, we consider the pair $(\ubar{\alpha}^r_k, \bar{\alpha}^r_k)$, but analogous argumentation applies also to $(\ubar{\alpha}_k, \bar{\alpha}_k)$.
If $\ubar{\alpha}^r_k \geq \tilde{\alpha}$, we accept $\bu_{rh}^r\big|_{[t_k,t_{k+1})} \coloneqq \bar \bu^r_k$ and apply it to the full system to obtain the next initial value. If $\ubar{\alpha}^r_k\leq  \tilde \alpha$, we reject the control, refine the model, and repeat the step. The refinement step is model-dependent and will be discussed in \cref{subsec:POD} for a POD-Galerkin ROM. A pseudocode of the ROM-RHC is outlined in \cref{algo:ROMRHC}. To prove its stability, we impose the following consistency assumption on the reduced space $V_r$, which will be verified for the POD-Galerkin ROM in \cref{lem:alpha_convergence}.
\begin{assumption}\label{ass:ROMconsistency}
    For each $k\in \N_0$, assume that we can make the ROM arbitrarily accurate with increasing dimension $r\in \N$ of $V_r$ in the sense that
    \begin{equation}\label{eq:ROM_accurate_property}
        \ubar\alpha_k^r \uparrow  \alpha(\bar \bu_k,t_k,\tilde y_k)
        \ (r\to \infty),
    \end{equation}
    where  $\alpha(\bar \bu_k,t_k,\tilde y_k)$ (defined in \cref{def:current_perfind}) is the FOM performance along the trajectory $\tilde y_{rh}$ with $\bar \bu_k \coloneqq \bar \bu (t_k,\tilde y_k)$ being the solution of \eqref{eqn:fin_hor_ocp} for $(\tint,\yint)=(t_k,\tilde y_k)$.
\end{assumption}
%
\begin{algorithm}[htbp!]
\caption{ROM-RHC($\delta,T, \tilde\alpha$)}\label{algo:ROMRHC}
\begin{algorithmic}[1]
\REQUIRE{Final time $T_{\infty} \in \mathbb{R}_{\geq 0} \cup \{ \infty\}$, sampling time $\delta>0$,  prediction horizon $T\geq \delta$, desired suboptimality $\tilde\alpha\in(0,\alpha)$, initial value $y_0\in H$;}
\ENSURE{ROM-RHC~$\mathbf{u}^r_{rh}$, non-decreasing sequence~$\{ t_k \}_{k\in\mathbb{N}}$.}
\STATE Set~$(\tint,\tildeyint)\colonequals(0,y_0)$, $ \tilde y_{rh}(\tint)\coloneqq y_0$, $t_0\coloneqq 0$; $k\coloneqq 0$; $V_r\coloneqq \{0\}$;
\WHILE{$ \tint <T_\infty$}\label{algoline:1}
\STATE \label{step4} Find the solution $(\bar { \bu}^r(\cdot\,;\tint,\yint),\bar { y}^r(\cdot\,;\tint,\yint))$ by solving~\eqref{eqn:reducedStabilizationProblem};
\STATE Compute $(\ubar \alpha_k^r,\bar \alpha_k^r)$ according to \eqref{eq:ROMsuboptimality_suff_fullROM} and \eqref{eq:ROMsuboptimality_necc_fullROM}, respectively;
\IF{$\ubar \alpha_k^r \geq  \tilde \alpha$}
\STATE Set $k\gets k+1$;  $t_k \gets \tint +\delta$;
\STATE For all $\tau\in[\tint,t_k)$, set $\bu^r_{rh}(\tau)\coloneqq \bar { \bu}^r(\tau;\tint,\tildeyint)$, and $\tilde y_{rh}(\tau)\coloneqq y(\tau;\bar{ \bu}^r,\tint, \tildeyint)$;
\STATE Update $(\tint,\tildeyint)\leftarrow(t_k, y(t_k;\bar{ \bu}^r,\tint, \tildeyint))$;
\ELSE
\STATE\label{algoline:2} Discard $\bar { \bu}^r( \cdot;\tint,\tildeyint)$, {refine model $V_r$ to satisfy \cref{ass:ROMconsistency} (cf. \cref{rem:pod_update} for a POD-ROM)};
\ENDIF
\ENDWHILE
\end{algorithmic}
\end{algorithm}
\begin{theorem}[Stability/suboptimality of ROM-RHC]\label{theorem:StabilityReducedRHC}
    Let \cref{ass:pde,ass:running_cost} and \cref{ass:ROMconsistency} hold and $ \alpha\in (0,1)$, $T\geq \bar T>\delta>0$ be chosen as in \cref{theorem:stabilityRHC}. Further, let $y_0\in H$ hold and $\tilde \alpha\in (0,\alpha)$ be a fixed desired performance.
     Then, for all $k\in \N_0$ there {exists a reduced space $V_{r_k}$ with dimension} $r_k\in \N$ with $0<\tilde \alpha \leq \ubar\alpha_k^{r_k}<\alpha$. {For these basis sizes}, the ROM-RHC $\bu^r_{rh}$ from \cref{algo:ROMRHC} satisfies the suboptimality inequality 
    \begin{equation}
        \label{eq:ROM_suboptimaity_inequality}
    V_\infty(0,y_0)\leq J_\infty(\bu^r_{rh};0,y_0)\leq \tfrac{1}{\tilde \alpha}V_T(0,y_0)\leq \tfrac{1}{\tilde \alpha}V_\infty(0,y_0).
    \end{equation}
    If additionally, \cref{ass:1} is satisfied, we have exponential stability of the closed-loop law, that is,
    \begin{equation}
        \label{eq:ROM_exponential_stability}
        \norm{\tilde y_{rh}(t)}{H}^2 \leq \tilde C_{rh} e^{-\tilde \zeta t}\norm{y_0}{H}^2 \quad  \text{ for }t\geq 0,
    \end{equation}
    where the positive numbers $(\tilde \zeta,\tilde C_{rh})$ depend on $(\tilde \alpha,\delta,T)$.
\end{theorem}
\begin{proof}
    Note that it holds $ \alpha(\bar \bu_k,t_k,y_k)\geq \alpha$ for all $k\in\N_0$, since otherwise this would be a contradiction to \cref{theorem:stabilityRHC}. Then \eqref{eq:ROM_accurate_property}, implies that 
    $\ubar \alpha_k^r\uparrow \alpha(\bar \bu_k,t_k,y_k)\geq \alpha>\tilde \alpha$ as $r\to \infty$. Therefore, a dimension $r_k\in \N$ with $\ubar \alpha_k^{ r_k}\geq \tilde \alpha$ exists. 
    For $r=r_k$, it holds $\tilde\alpha \leq \ubar \alpha_k^{{r_k}}$ for all $k\in \N_0$. Rewriting this inequality using the definition of $\ubar{\alpha}_k^r$ from \eqref{eq:ROMsuboptimality_suff_fullROM}, together with the error estimator properties from \cref{cor:errorest_optimalval_diff_initguess,cor:costbound}, yields
    \begin{align*}
    &\tilde\alpha J_\delta(\bar \bu^{{r_k}}_k;t_k,\tilde y_k)\leq
         \tilde\alpha \big(J^{{r_k}}_\delta(\bar \bu_k^{{r_k}};t_k,\tilde y_k)+\Delta_{J_\delta}(t_k, \tilde y_k, \bar \bu_k^{{r_k}}, 0 )\big)\\
         &\quad\leq V_T^{{r_k}}(t_k, \tilde y_k)-\Delta_{V_T}(t_k,\tilde y_k,0)-V_T^{{r_k}}(t_{k+1},y^{{r_k}}(t_{k+1};\bar \bu_k^{{r_k}},t_k,\tilde y_k))\\
         &\qquad-\Delta_{V_T}(t_{k+1},y^{{r_k}}(t_{k+1};\bar \bu_k^{{r_k}},t_k,\tilde y_k),\Delta_{\tilde y_{k+1}})\leq V_T(t_{k},\tilde y_k)-V_T(t_{k+1},\tilde y_{k+1})
    \end{align*}
    for all $k\in \N_0$. Observe that $\tilde y_k = \tilde y_{rh}(t_k)$, so that \eqref{eq:RelaxedDDP_suboptimaity_inequality} is satisfied with $\bu = \bu^r_{rh}$ and $y = \tilde y_{rh}$. The claim then follows directly from \cref{theorem:RelaxedDDP}.
\end{proof}
Using the triangle inequality, the following result can be shown.
\begin{corollary}[Exponential decay of the error]
    In the situation of \cref{theorem:stabilityRHC,theorem:StabilityReducedRHC}, we have    
    \begin{equation}\nonumber
        \norm{y_{rh}(t;0,y_0)-\tilde y_{rh}(t;0,y_0)}{H}^2\leq \hat Ce^{-\hat \zeta t}\norm{y_0}{H}^2 \quad  \text{ for }t\geq 0,
    \end{equation}
    where the positive numbers $(\hat \zeta,\hat C)$ depend on $(\alpha, \tilde\alpha,\delta,T)$.
\end{corollary}
\begin{remark}\label{rem:perfrind}
    If the performance estimator $\ubar \alpha_k^r$ is employed to construct the ROM-RHC $\bu_{rh}^r$, the results of \cref{theorem:StabilityReducedRHC} also apply to the reduced cost $J^r_\infty$ in place of $J_\infty$. Consequently, using $\ubar \alpha_k^r$ stabilizes both the FOM and the ROM simultaneously. {The performance conditions
\eqref{eq:ROMsuboptimality_suff} or $\ubar{\alpha}_k^r \ge \tilde\alpha$
can be interpreted as balancing the error in the decrease of the two subsequent value functions relative to the cost along the controlled trajectory. These conditions are local in the sense that the model space and its
dimension may vary with $k$.}
\end{remark}

\subsection{POD-Galerkin model reduction}\label{subsec:POD}
In the following, we show how \cref{algo:ROMRHC} can be implemented using orthonormal bases generated by POD. As a first step, we verify \cref{ass:ROMconsistency}.
\begin{lemma}
    \label{lem:alpha_convergence}
    Let $(v_n)_{n\in \N}$ be a complete orthonormal basis of the separable Hilbert space $V$ and let $V_r=\text{\emph{span}}\{v_1,\ldots,v_r\}$.
    Then,  \cref{ass:ROMconsistency} is satisfied.
\end{lemma}
\begin{proof}
    Let $\bar x_k^r=(\bar y^r_k,\bar\bu^r_k,\bar p^r_k)$ and $\bar x_k=(\bar y_k,\bar\bu_k,\bar p_k)$ be the solution of \eqref{eq:ROML} and \eqref{eq:FOML}, respectively, for $(\tint,\yint)=(t_k,\tilde y_k)$. Note that by definition, we have
    \begin{equation}\nonumber
         \alpha(\bar \bu_k,t_k,\tilde y_k)\coloneqq \frac{ V_T(t_k, \tilde y_k) -  V_T(t_{k+1}, y(t_{k+1};\bar \bu_k,t_k,\tilde y_k))}{J_\delta(\bar \bu_k;t_k,\tilde y_k) },
    \end{equation}
    and the lower bound $\ubar{\alpha}_k^r$ is defined in \eqref{eq:ROMsuboptimality_suff_fullROM}. Thus, to prove \cref{ass:ROMconsistency}, we have to show convergence of the error and its estimators at time steps $t_k, t_{k+1}$. First, we consider the time step $t_k$ without an error in the initial condition $\tilde y_k$. By \cref{theo:os_conv}, we have $\bar x^r \to \bar x$ in $\mX_T(\tint)$ as $r\to \infty$. Hence, by triangle inequality, $y^r(\bar \bu^r_k,t_k,\tilde y_k)\to y(\bar \bu^r_k,t_k,\tilde y_k)$, $p^r(\bar y^r_k)\to p(\bar y^r_k)$ in $\mY_T(t_k)$, and therefore by the error estimator equivalence \cref{lemma:app_error_equivalence}
     $\Delta_y(0,0,\bar y^r_k,\bar \bu_k^r)$, $\Delta_p(0,\bar y^r_k,\bar p^r_k)\to 0 \ (r\to \infty)$. Therefore, by \cref{theo:oc_est}, $\bar \Delta_{y,H}(\bar x_k^r,0)\to 0$, and $\bar \Delta_\bu(\bar x_k^r,0)\to 0$ and also $\bar \Delta_{y}(t_k,\tilde y_k)$, $\bar \Delta_{p}(t_k,\tilde y_k)\to 0 $ by their definition in \eqref{eq:optimal_error_ests1} in \cref{cor:errorest_optimalval_diff_initguess}. Therefore $\Delta_{V_T}(t_k,\tilde y_k,0)\to 0$ and $\Delta_{J_\delta}(t_k, \tilde y_k, \bar \bu^r_k, 0 )\to 0$ by definition in \eqref{eq:valueest_line1} in \cref{cor:errorest_optimalval_diff_initguess} and \cref{cor:costbound}, respectively. Consider now time step $t_{k+1}$ and let $\bar x_{k+1}^r=(\bar y^r_{k+1},\bar\bu^r_{k+1},\bar p^r_{k+1})$ be the solution of \eqref{eq:ROML} for $(\tint,\yint)=(t_{k+1},\bar y^r_k(t_{k+1}))$ and $\bar x_{k+1}=(\bar y_{k+1},\bar\bu_{k+1},\bar p_{k+1})$ be the solution of \eqref{eq:FOML} for $(\tint,\yint)=(t_{k+1},\bar y_k(t_{k+1}))$. Hence, there is an error in the initial condition that decays using the argumentation for step $t_k$. Therefore, one can apply similar arguments to show that $\Delta_{V_T}(t_{k+1},y^r(t_{k+1}),\Delta_{\tilde y_{k+1}})\to 0$ as $r\to \infty$.
\end{proof}

In practice, orthonormal bases that are $L^2(t_k,t_k+T;V)$-optimal w.r.t. to a given snapshot set $S\subset L^2(t_k,t_k+T;V)$ can be constructed by POD via the minimization problem
\begin{align}
    \label{eqn:PODminimization}
     \min_{V_r=\linspan(v_i)_{i=1}^r\subseteq V} \sum_{s\in S}\norm{s-\Pi^V_{V_r}s }{L^2(t_k,t_k+T;V)}^2\mathrm dt
     \ \  \text{s.t.}\ \ \langle v_i, v_j\rangle_V = \delta_{ij} \ (i,j = 1,\ldots, r),
\end{align}
where $\Pi^V_{V_r}v\coloneqq \sum_{i=1}^{r} \langle  v, v_i\rangle_V v_i $ for $v\in V$.
The solution to \eqref{eqn:PODminimization} is called the POD basis $(v_i)_{i=1}^{r}$ of rank $r\leq r_{\max} \in \N \cup \{\infty \}$ and can be characterized as eigenfunctions of the correlation operator (cf. \cite{GubV17}). In the following theorem, we state that choosing the FOM solution {of \eqref{eqn:fin_hor_ocp} at time instants $t_k$ and $t_{k+1}$} as snapshots either yields an orthonormal basis of \( V \), or leads to an interpolation of the full-order performance index {$\alpha(\bar \bu_k,t_k,\tilde y_k)$} using a finite-dimensional subspace. 
\begin{lemma}[POD model update]\label{lem:remark_alpha_convergence}
    Let $V_{r_{\text{\tiny old}}}$ be the POD space constructed with the snapshots set $S_{\text{\tiny old}}\subset H$ with maximal rank $r_{\text{\tiny old}}\in \N$. Suppose at step $k\in \N$, it holds, $\ubar \alpha_k^{r_{\text{\tiny old}}}< \tilde\alpha$ {with $\tilde\alpha\in (0,\alpha)$ chosen as in \cref{theorem:StabilityReducedRHC}}. Then, construct $V_{r_{\text{\tiny new}}}$ by POD with maximal rank $r_{\text{\tiny new}} \in \N \cup \{\infty\}$ with the new snapshot set $S_{\text{\tiny new}}\coloneqq  S_k\cup S_{k+1}\subset \mY_T(t_k)$ with $S_k=\{\bar y_k, \bar p_k\}$ from the proof of \cref{lem:alpha_convergence}, $S_{k+1} =\{\bar y_{k+1}(\cdot-\delta), \bar p_{k+1}(\cdot-\delta)\}$, where $\bar x(t_{k+1},\bar y_k(t_{k+1}))=(\bar y_{k+1},\bar \bu_{k+1},\bar p_{k+1})$
    solve \eqref{eq:FOML} for $(\tint,\yint)=(t_{k+1},\bar y_k(t_{k+1}))$. If $r_{\text{\tiny new}}=\infty$, then $V_{r_{\text{\tiny new}}}=V$ with complete orthonormal basis $(v_r)_{r\in \N}$. If $r_{\text{\tiny new}}\in \N$, then 
     $\ubar \alpha_k^{r_\text{\tiny new}} = \alpha(\bar \bu_k,t_k,\tilde y_k)>\tilde \alpha$.
\end{lemma}
\begin{proof}
    If $\ubar \alpha_k^{r_{\text{\tiny old}}} < \tilde\alpha$, then at least one of the quantities $\bar y_k$, $\bar p_k$, $\bar y_{k+1}$, or $\bar p_{k+1}$ is not contained in $\mY_T^r(\tint)$, since otherwise, by the interpolation property of the optimality system from \cref{lem:interpolation_properties_2} in \cref{lem:interpolation_properties} {at time steps $k$ and $k+1$}, we would have {$\bar y_k=\bar y_k^r$, $\bar p_k=\bar y_k^r$ and $\bar y_{k+1}=\bar y_{k+1}^r$, $\bar p_{k+1}=\bar y_{k+1}^r$.  Thus, by the definition of $\ubar \alpha_k^{r}$ in \eqref{eq:ROMsuboptimality_suff_fullROM} and  $\alpha(\bar \bu_k, t_k, \tilde y_k)$ in \eqref{eq:current_pfind} it would hold} 
    \(
    \ubar \alpha_k^{r_{\text{\tiny old}}} = \alpha(\bar \bu_k, t_k, \tilde y_k) \geq \alpha > \tilde \alpha,
    \)
    which contradicts the assumption {$\ubar \alpha_k^{r_{\text{\tiny old}}}< \tilde \alpha$}. {Hence, there is new snapshot information in the set $S_{\text{\tiny new}}$. Note that by the time scaling of the snapshots in $S_{k+1}$ it really holds $S_{\text{\tiny new}}\subset \mY_T(t_k)\subset {L^2(t_k,t_k+T;V)}$.} Thus, we  can solve \eqref{eqn:PODminimization} for $S=S_{\text{\tiny new}}$ to obtain the new space $V_{r_{\text{\tiny new}}}$. If $r_{\text{\tiny new}} = \infty$, then by the constraint in \eqref{eqn:PODminimization}, the POD basis forms a complete orthonormal basis of $V$ {and we are in the situation of \cref{lem:alpha_convergence}}. If instead $r_{\text{\tiny new}} \in \mathbb{N}$, then all snapshots are reproduced exactly in $V_{r_{\text{\tiny new}}}$, so that $s \in \mY_T^r(t_k)$ for all $s \in S_{\text{\tiny new}}$. {Therefore, using the same argument as before,} we interpolate the exact solution of the optimality systems at $t_k$ and $t_{k+1}$ by \cref{lem:interpolation_properties_2} in \cref{lem:interpolation_properties}, which yields
    \(
    \ubar \alpha_k^{r_{\text{\tiny new}}} = \alpha(\bar \bu_k, t_k, \tilde y_k)>\tilde \alpha
    \). {Thus, \eqref{eq:ROM_accurate_property} in \cref{ass:ROMconsistency} is satisfied with equality for a finite $r=r_{\text{\tiny new}}$.} 
\end{proof}
{
\begin{remark}\label{rem:pod_update}
Hence, to update the POD-ROM in Line~\ref{algoline:2} of \cref{algo:ROMRHC}, we invoke
\cref{lem:remark_alpha_convergence}, compute the snapshot set
\(S_{\text{\tiny new}}\), and solve \eqref{eqn:PODminimization} with
\(S = S_{\text{\tiny new}}\). That is, we construct a combined reduced basis for the state and adjoint variables to be able to use the interpolation property \cref{lem:interpolation_properties} in the proof of \cref{lem:remark_alpha_convergence}. With this basis choice, \cref{lem:remark_alpha_convergence} guarantees that the prescribed performance criterion is satisfied when selecting the maximal POD rank \(r_{\max}\).
We emphasize that the choice \(r = r_{\max}\) is primarily of theoretical interest to
verify \cref{ass:ROMconsistency}. In practice, the reduced basis dimension is typically
chosen more heuristically, for instance by requiring that a prescribed
fraction of the snapshot energy is captured (cf.~\cite{GubV17}).
Moreover, the snapshot set \(S_{k+1}\) is introduced solely for theoretical purposes to
prove \cref{lem:remark_alpha_convergence} and was not required for the numerical
experiments presented in \cref{sec:NUMEXP}.
Since the computation of optimal snapshots constitutes the most time-consuming part of
\cref{algo:ROMRHC}, a promising extension is to combine the proposed approach with adaptive
optimization strategies, enabling efficient (possibly real-time) basis construction and
simultaneous optimization.
\end{remark}
}
\section{Concrete example}\label{sec:example}
Here, we present an example within the proposed framework for which \cref{ass:pde,ass:running_cost,ass:1} are satisfied (cf. \cite{azmi2019hybrid}). For $n\in \N$, consider a bounded Lipschitz domain $\Omega\subset\R^n$ and \eqref{eqn:infhorizonproblem} with
\begin{equation}
\label{eq:cost_example}
J_{\infty}(\mathbf{u};0,y_0):= \frac{1}{2}\int_{0}^{\infty}\norm{y(t)}{L^2(\Omega)}^2+\tfrac{\lambda}{2} |\mathbf{u}(t)|^2_{2}+\tfrac{\beta}{2} |\mathbf{u}(t)|^2_{1}\ \dt
\end{equation}
for $\lambda,\beta >0$, $\norm{\cdot}{2}$ being the Euclidian norm on $\R^m$, and $|\mathbf{u}(t)|_{1}=\sum_{i=1}^m|\bu_i(t)|$.             
Further, consider \eqref{eq:LTV} as
{\begin{equation}
    \label{eq:pde_example}
    \begin{aligned}
        \partial_t y(t,x)-\nu\Delta y(t,x) + a(t,x)y(t,x)+ b(t,x) \cdot \nabla y(t,x)&=\textstyle \sum\limits^m_{i =1} \mathbf{1}_{R_i}(x)\bu_i(t)&& (t,x)\in Q\coloneqq(0,\infty)\times\Omega,\\
          y(t,x)&=0&&(t,x)\in (0,\infty)\times \partial \Omega,\\
        y(0,x)&=y_0(x)&& x\in  \Omega,
    \end{aligned}
\end{equation}}
for $\nu>0$, $y_0\in L^2(\Omega)$, $a\in L^\infty(Q)$, $b\in L^\infty(Q; \R^n) $ with $\nabla\cdot b\in L^\infty(Q) $. Moreover, the actuators are chosen as indicator functions on rectangular subdomains \( R_i \subset \Omega \) (cf.~\cref{fig:ErrorEstOL}) for $i=1,\ldots,m$. 
Defining $U=\R^m$, $H=L^2(\Omega)$, $V=H_0^1(\Omega)$, {$A(t)=-\nu\Delta y(t) + a(t)y(t)+ b(t)\cdot\nabla y(t)\in \calL(V,V')$}, $B(t)=[\mathbf{1}_{R_1}, \ldots,\mathbf{1}_{R_m}]\in \calL(\R^m,V')$, $g(\bv)=\tfrac{\beta}{2} |\bv|^2_{1}$ for $\bv\in U$, we are in the situation of \cref{sec:FOM_stability_RHC} and the assumptions are satisfied.
\begin{lemma}
     \cref{ass:pde,ass:running_cost} are satisfied for \eqref{eq:cost_example}-\eqref{eq:pde_example} with $\eta_V = \nu$ and $\eta_H=\essinf\{a(t,\bx)-\nicefrac{1}{2}(\nabla \cdot b)(t,\bx)\,|\,(t,\bx)\in Q\}$. Furthermore, if the number of actuators $m\in\N$ is chosen sufficiently large, then \cref{ass:1} is satisfied.
\end{lemma}
\begin{proof}
    The boundedness, continuity, and weak coercivity of the operators \( (A, B) \) follow directly from the regularity of \( (a, b) \) in combination with standard estimates. Furthermore, using similar arguments as in \cite[Section~4.1]{azmi2019hybrid}, one can show that \( g \) satisfies \cref{ass:running_cost}. The assertion in \cref{ass:1} follows by analogous reasoning to that in \cite[Proposition~6.1]{azmi2019hybrid}. Note that the argumentation presented there applies to all \( g \) satisfying \( g(\bv) \leq C \norm{\bv}{2}^2 \) for \( \bv \in U \), where \( C > 0 \) is a constant independent of  \({\bv}, T \) and \( \tint \).
\end{proof}
\section{Numerical experiments}\label{sec:NUMEXP}
In this section, we compare \cref{algo:FOMRHC} (FOM-RHC) and \cref{algo:ROMRHC} (ROM-RHC) in terms of closed-loop stability and computational performance. 
\subsection{Algorithmic setup and discretization}\label{subsec:NUMEXP_setup}
For the example from \cref{sec:example}, we choose the problem parameters listed in \Cref{tab:parameters} together with
\begin{equation*}
    a(t,\bx) \coloneqq -2-0.8\,|\sin(t)|,\quad  b(\bx)\coloneqq (-0.01(x_1+x_2),0.2x_1x_2)^\top
\end{equation*}
for $t\in(0,\infty)$ and $\bx=(x_1,x_2)\in\Omega$. To obtain a numerical model, we consider a discretization in space using piecewise linear finite elements with \( 3721 \) degrees of freedom. We fix the final time to \( T_\infty = 10 \). For the discretization in time, we apply an implicit Euler scheme using \( K = 801 \) time points with step size \( \tau = \nicefrac{T_\infty}{(K-1)} = 0.0125 \). The sampling time is chosen as \( \delta = {0.25} \), while \( T \) and \( \tilde \alpha \) are varied in the experiments below.
For the model reduction, we employ POD \eqref{eqn:PODminimization} {together with the strategy outlined in \cref{rem:pod_update} for the choice of snapshots and basis size. The corresponding tolerance for the contained energy was chosen tightly as \( \varepsilon_{\tiny \mathsf{POD}} = 1 - 10^{-13} \), which mimics using the maximal POD rank. As snapshots for basis construction at step $k$, we use the optimal state and adjoint state at time step $k$ together with previously used snapshots.} Note that, for the above choice of \( a \) and \( b \), time and space variables are affinely separable. Consequently, an offline-online decomposition of the error estimator is considered (cf.~\cite{hesthaven2016certified}). {However, if $a,b$ are not affinely separable, one can either use hyper-reduction methods such as DEIM or consider computing the error estimators directly, which, in turn, would need FOM solves.} To solve the non-smooth finite-horizon control problems, we use the Barzilai–Borwein proximal gradient method from \cite{azmi2025forward} using both absolute and relative tolerances of \( 10^{-13} \). We consider \( m = 13 \) square actuators, each with area \( 0.0106 \), arranged in an L-shape as illustrated in \Cref{fig:ErrorEstOL}. The total actuator area thus covers approximately \( 14\% \) of the domain.
\begin{table}[ht!]
    \centering
    \caption{Problem parameters for the numerical setup.}
    \label{tab:parameters}
    \begin{tabular}{rcccccccc}
        \toprule
        \textbf{Parameter} &$\Omega$& $m$  & $\nu$ & $y_0$ & $\lambda$ & $\beta$ \\\midrule
        \textbf{value} & $(0,1)^2$& $13$ & 0.1 & $3\sin(\pi x_1)\sin(\pi x_2)$ & $10^{-3}$ & $10^{-4}$\\\bottomrule
    \end{tabular}
\end{table}
\subsection{Error estimation for the open-loop problem}\label{subsec:openloop}
First, we investigate the error estimators introduced in \cref{sec:ROM} to numerically verify \cref{cor:errorest_optimalval_diff_initguess} and \cref{lem:errorest_optimalval_diff_initguess_prepare} for the first open-loop problem \eqref{eqn:fin_hor_ocp} for $(\tint,\yint)=(0,y_0)$. In \Cref{fig:ErrorEstOL}, we plot the decay of the value function estimators $\Delta_{V_T}$ and the true error $e_{V_T}\coloneqq|V_T(0,y_0)-V_T^r(0,y_0)|$ as a function of the reduced basis size $r$, where the reduced space is constructed by POD based on snapshots of the FOM optimal state and adjoint for different $L^2$-regularization parameters $\lambda \in \{1, 10^{-3} \}$ and $T=1$. An exponential decay (with the same rate) in $r$ is observed for all quantities, along with a consistent overestimation of the error. This overestimation is more significant for smaller $L^2$-regularizations due to the scaling factors $\nicefrac{1}{\lambda}$ appearing in \cref{eq:opt_cont_est} and \cref{eq:opt_out_est}. In \cref{fig:mpcExample1table}, we report the true error $e_{V_T}$, the error estimators $\Delta_{V_T}$ and their effectivities defined as $\text{eff}(\Delta_{V_T})\coloneqq \nicefrac{e_{V_T}}{\Delta_{V_T}}$, for prediction horizons $T\in \{ 0.8,1,1.2\}$ and regularizations $\lambda \in \{1, 10^{-1}, 10^{-2}, 10^{-3} \}$. These results are obtained using a reduced space of dimension $r=30$ again constructed from the optimal state and adjoint corresponding to each pair of $(T,\lambda)$. The results show that the error estimator increases by approximately one order of magnitude for each increase in the prediction horizon \( T \) and for each decrease in the regularization parameter \( \lambda \). The increase in dependence of $T$ is due to the exponential terms $C_{1,y}(t), C_{2,y}(t)$  appearing in \eqref{eq:state_apost_CHL2V}, and $C_p$ in \eqref{eq:adjstate_apost_CHL2V}. This effect is only marginally reflected in the true error, leading to effectivities between \( 2.0\cdot 10^{-3} \) for $(T,\lambda)=(0.8,1)$ and \( 5.2 \cdot 10^{-8} \) for $(T,\lambda)=(1.2,10^{-3})$. 

Thus, for large prediction horizons $T$ and small $L^2$-regularization,  the rigorous error estimator $\Delta_{V_T}$ increasingly overestimates the true error. This behavior is consistent with observations in \cite{dietze2023reduced,karcher2018certified}. {Since the proposed estimator exhibits an exponential dependence on the horizon length \(T\), leading to large effectivities, it should be interpreted as a worst-case bound. Practically, the estimator remains meaningful for small to moderate horizons and problems with larger $L^2$-regularization. Since both the true error and the estimator decay exponentially with respect to the reduced basis dimension, the resulting
conservativeness can be mitigated by a linear increase of the basis size or, for large
horizons in the RHC framework, by choosing \(\tilde{\alpha}\) in \cref{theorem:StabilityReducedRHC} closer to zero. In the following, we restrict ourselves therefore up to moderate horizons.}
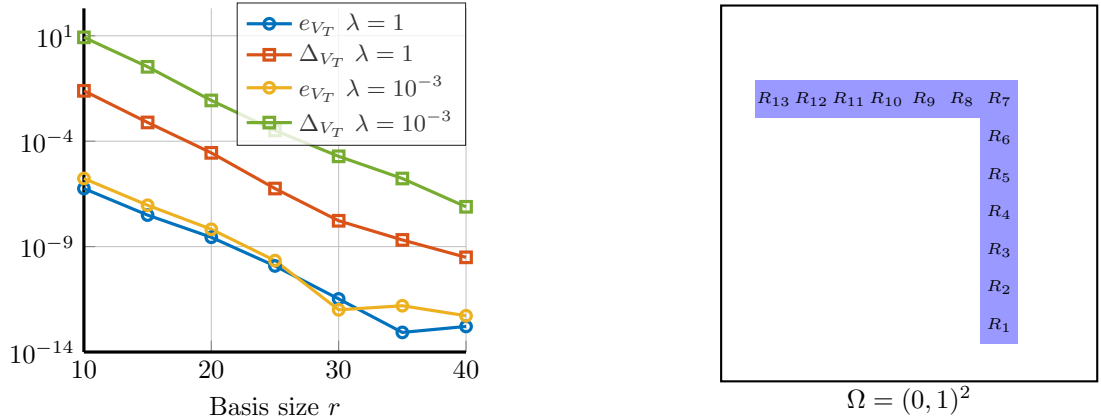
\begin{figure}[ht!]
	\centering
	\begin{subfigure}[t]{0.48\textwidth}
        \centering
        \includegraphics{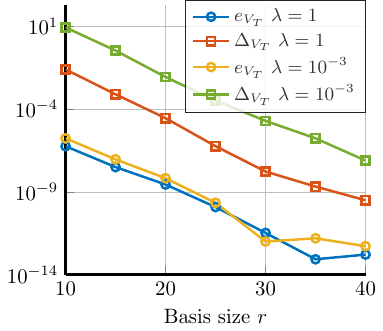}
	\end{subfigure}
 	\hfill
    \begin{subfigure}[t]{0.48\textwidth}
        \centering
    \begin{tikzpicture}[scale=5]

    \draw[thick] (0,0) rectangle (1,1);
    \node at (0.5,-0.051) {$\Omega=(0,1)^2$};
    
    \def\s{0.1}
    \def\yshift{0.1}
    \def\xshift{-0.01}
    
    \foreach \j [evaluate=\j as \k using int(7 - \j)] in {1,...,6} {
      \fill[blue!40] ({\xshift + 0.7}, {0.6 + \yshift - \j*\s}) rectangle ++(\s, \s);
      \node at ({\xshift + 0.7 + 0.5*\s}, {0.6 + \yshift - \j*\s + 0.5*\s}) {\tiny\sf $R_{\k}$};
    }
    
    \foreach \i [evaluate=\i as \k using int(13 - \i)] in {0,...,6} {
      \fill[blue!40] ({\xshift + 0.1 + \i*\s}, {0.6 + \yshift}) rectangle ++(\s, \s);
      \node at ({\xshift + 0.1 + \i*\s + 0.5*\s}, {0.6 + \yshift + 0.5*\s}) {\tiny\sf $R_{\k}$};
    }
    
    \end{tikzpicture}
        \label{fig:actuator_plot}
	\end{subfigure}
	\caption{Left: true error $e_{V_T}$, estimators $\Delta_{V_T}$, and effectivities $\text{eff}(\Delta_{V_T})$ for $T=1$, plotted against basis size $r$ for the first open-loop problem. Right: actuator setup.}
	\label{fig:ErrorEstOL}
\end{figure}
\begin{table}[ht!] 
    \scriptsize
	\centering 
    \caption{Error estimation for the value function of the first open-loop problem {for $r = 30$}.}
	\label{fig:mpcExample1table}

 \begin{tabular}{c*{9}{c}}

    \begin{tabular}{c*{9}{c}}
    & \multicolumn{3}{c}{$T = 0.8$} 
    & \multicolumn{3}{c}{$T = 1$} 
    & \multicolumn{3}{c}{$T = 1.2$}  \\
    \toprule
    $\lambda$ & $e_{V_T}$ & $\Delta_{V_T}$ & eff. 
              & $e_{V_T}$ & $\Delta_{V_T}$ & eff. 
              & $e_{V_T}$ & $\Delta_{V_T}$ & eff. \\
    \midrule
    $10^{0}$   & $6.1e{-13}$ & $3.2e{-10}$ & $2.0e{-3}$ 
              & $2.6e{-12}$ & $2.2e{-9}$  & $1.2e{-3}$ 
              & $3.3e{-12}$ & $1.7e{-8}$  & $2.0e{-4}$ \\
    $10^{-1}$  & $1.3e{-12}$ & $1.5e{-9}$ & $9.1e{-4}$ 
              & $5.0e{-12}$ & $1.4e{-8}$  & $3.6e{-4}$ 
              & $1.1e{-11}$ & $1.0e{-7}$  & $1.0e{-4}$ \\
    $10^{-2}$  & $1.3e{-12}$ & $2.4e{-8}$ & $5.6e{-5}$ 
              & $2.7e{-13}$ & $2.8e{-7}$  & $9.8e{-7}$ 
              & $3.9e{-12}$ & $2.7e{-6}$  & $1.4e{-6}$ \\
    $10^{-3}$  & $4.6e{-12}$ & $2.0e{-7}$ & $2.3e{-5}$ 
              & $2.4e{-12}$ & $3.1e{-6}$  & $7.9e{-7}$ 
              & $1.0e{-12}$ & $1.9e{-5}$  & $5.2e{-8}$ \\
    \bottomrule
\end{tabular}
\end{tabular}
\end{table}
\subsection{Stability and suboptimality of the RHC algorithms}\label{subsec:numstab}
We now investigate the performance of the full-order and reduced-order RHC schemes. 
In \Cref{fig:Decay}, we depict the decay behavior of the reduced-order schemes over time for $(T,\tilde{\alpha}) \in \{ (0.8,0.35), (1.0,0.58), (1.2,0.73) \},$
where the minimal desired performance $\tilde{\alpha}$ is chosen such that $\tilde{\alpha} \leq \alpha$. Here, the FOM performance index $\alpha \coloneqq \min_k \alpha_k^{\text{\tiny FOM}}$ is defined by $\alpha_k^{\text{\tiny FOM}} \coloneqq \alpha(\bu_{rh}, t_k, y_{rh}(t_k)),$ see \eqref{eq:current_pfind}, for the output $\bu_{rh}^{\text{\tiny FOM}}, y_{rh}^{\text{\tiny FOM}}$ generated by the FOM-RHC scheme in \cref{algo:FOMRHC}. {Note that $\alpha$ as defined here is the numerically observed performance index and therefore different from the theoretical performance index \eqref{e3}, which we do not aim to compute here.} Further, we denote the output of the ROM-RHC scheme (cf.~\cref{algo:FOMRHC}) by $\bu_{rh}^{\text{\tiny ROM}}, y_{rh}^{\text{\tiny ROM}}$, using the performance estimators $\ubar{\alpha}_k$ (defined in \eqref{eq:ROMsuboptimality_suff}) or $\ubar{\alpha}_k^r$ (defined in \eqref{eq:ROMsuboptimality_suff_fullROM}). From \Cref{fig:Decay}, we observe that all RHC schemes exponentially stabilize the system, as predicted by \cref{theorem:stabilityRHC,theorem:StabilityReducedRHC}. Moreover, a larger prediction horizon $T$ yields a higher FOM performance $\alpha$ (and thus potentially allows for a larger choice of $\tilde{\alpha}$), which in turn results in a faster exponential decay of both the state and the cost. The differences between the ROM and FOM methods are notably small. This is also reflected in \Cref{fig:RHC_error_comp}, where we compare the schemes in terms of the approximation quality for the cost over the entire time horizon,
$ J^{\star}_{T_\infty} \coloneqq J_{T_\infty}(\bu^{\star}_{rh})$ for $ \star \in \{\text{FOM}, \text{ROM}\},$ as well as in terms of the corresponding relative errors
\begin{equation}\nonumber 
   e_{J_{T_\infty}} \coloneqq \frac{|J^{\text{\tiny ROM}}_{T_\infty}- J^{\text{\tiny FOM}}_{T_\infty}|}{J^{\text{\tiny FOM}}_{T_\infty}}, \quad e_{y_{rh}}\coloneqq  \frac{\norm{y_{rh}^{\text{\tiny ROM}}- y_{rh}^{\text{\tiny FOM}}}{L^2(0,T_\infty;H)}}{\norm{y_{rh}^{\text{\tiny FOM}}}{L^2(0,T_\infty;H)}}, \quad e_{\bu_{rh}}\coloneqq  \frac{\norm{\bu_{rh}^{\text{\tiny ROM}}- \bu_{rh}^{\text{\tiny FOM}}}{\mU_{T_\infty}}}{\norm{\bu_{rh}^{\text{\tiny FOM}}}{\mU_{T_\infty}}}.
\end{equation}
Furthermore, we present the results of the ROM-RHC scheme using the performance estimators $\ubar{\alpha}_k$ and $\ubar{\alpha}_k^r$. The corresponding relative error in the performance index is defined as $e^{(r)}_{\alpha} \coloneqq \nicefrac{\left| \min_k\ubar{\alpha}^{(r)}_k - \alpha \right|}{\alpha}.$ Across all test cases, the relative error consistently remains in the range of $10^{-5}$ to $10^{-8}$, indicating that the RHC schemes behave almost identically in all considered scenarios. In \Cref{fig:RHC_perfromacne}, we compare the computational cost of the RHC schemes. The reduced variants achieve a speed-up of approximately $10$-$13$ in computation time compared to the full-order RHC, corresponding to a reduction in the total number of FOM gradient evaluations by a factor of about $20$. For the reduced methods, FOM gradient evaluations are needed only if a model update is triggered, to compute the full optimal state and adjoint at the current time step $t_k$, which serve as snapshots for updating the new ROM (cf.~\cref{lem:remark_alpha_convergence}). \Cref{fig:RHC_perfromacne} reveals slightly higher speed-ups for the estimator $\ubar{\alpha}_k$ compared to $\ubar{\alpha}_k^r$. This is because, for $\ubar{\alpha}_k$, the optimization result from time step $t_{k+1}$ can be cached when computing the performance index (see \eqref{eq:ROMsuboptimality_suff}). Such caching is not possible for $\ubar{\alpha}_k^r$, because the initial value of the value function at time $t_{k+1}$ differs (see \eqref{eq:ROMsuboptimality_suff_fullROM}). In this case, however, the optimization result at $t_{k+1}$ is still used as a warm start for the next optimization, leading only to a slightly increased computation time. 
Both ROM-RHC variants show comparable behavior w.r.t. the resulting reduced basis size $r$ and the number of model updates. In all cases, the model updates are completed within the first three iterations. 

\Cref{fig:Perf} illustrates the evolution of the performance indices throughout the RHC iterations $k$ for the choice $(T,\tilde{\alpha}) = (0.8, 0.35)$. One observes in the right plot that the performance estimator $\ubar{\alpha}_k^r$ has a larger overestimation than $\ubar{\alpha}_k$, due to the additional terms depending on the initial value error in \cref{cor:errorest_optimalval_diff_initguess} and the cost function error estimate in the denominator in \eqref{eq:ROMsuboptimality_suff_fullROM}. {Eventhough the numerical experiments suggest that $\ubar \alpha_k$ is more efficient to use, we note that, combined with an offline-online decomposition, however, the verification using $\ubar \alpha_k^r$ becomes entirely independent of the full-order model, which can be advantageous for very large or frequently updated problems. Finally, the performance index $\ubar \alpha_k^r$ also offers theoretical benefits, as it enables simultaneous stabilization of the full- and reduced-order model
(see \cref{rem:perfrind}).} After the second model update, all performance indices approximately match the FOM performance index. 

{In \Cref{fig:mpcest}, we plot the error estimators for the optimal state, control, and adjoint (defined in \eqref{eq:opt_cont_est} and \eqref{eq:optimal_error_ests1}) together with the corresponding true errors as functions of the RHC iteration $k$, for both performance index variants.
For visualization purposes, the refined error estimators (and true errors) are not plotted at model update steps, since they are set to zero whenever the FOM optimal control is available. Similar to the value function estimator discussed in \cref{subsec:openloop}, one observes a pronounced initial overestimation and a rapid increase of both the error estimators and the true error. At $k=2$, model refinement is triggered, leading to a reduction of both quantities by approximately two orders of magnitude. During subsequent RHC iterations, as the controlled state decays to zero, the true error also converges to zero. This asymptotic behavior is accurately captured by the error estimators, which likewise decay.
}

Moreover, the sparsity pattern of the optimal control is accurately captured by the reduced-order methods, and in all cases, the actuators $R_1$ and $R_{13}$ remain inactive. 
Overall, for this example, the ROM-RHC schemes achieve low errors while providing substantial savings in computational cost.
\begin{figure}[htbp]
	\centering
	\begin{subfigure}[t]{0.48\textwidth}
        \centering
	    \includegraphics{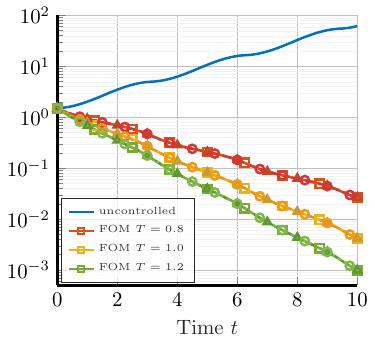}
        \label{fig:Decay_sub1}
	\end{subfigure}
 	\hfill
    \begin{subfigure}[t]{0.48\textwidth}
        \centering
	    \includegraphics{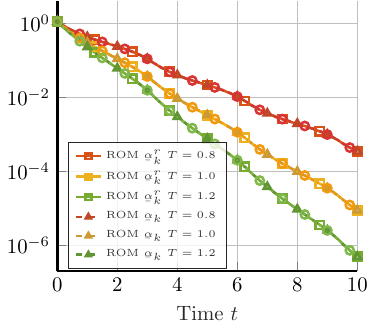}
        \label{fig:Decay_sub2}
	\end{subfigure}
	\caption{Exponential decay of the state norm $|y_{rh}^\star(t)|_H$ (left) and the cost $\ell(y_{rh}^\star(t),\bu_{rh}^\star(t))$ (right) for $\star \in \{\text{FOM}, \text{ROM}\} $ for the different performance estimators.}
	\label{fig:Decay}
\end{figure}
%
\begin{table}[ht!] 
    \scriptsize
	\centering 
    \caption{Error comparison of the RHC schemes.}
	\label{fig:RHC_error_comp}
	  \begin{tabular}{lcccccccc}\toprule
		$(T,\tilde \alpha)=(0.8,0.35)$ & $J^\star_{T_\infty}$ & $|y^\star_{rh}(T_\infty)|_H$ & $e_{J_{T_\infty}}$  & $e_{\bu_{rh}}$ & $e_{y_{rh}}$   & min/avg/max($\ubar\alpha_k$) & $e^{(r)}_\alpha$ \\ 
        \midrule
		FOM-RHC & 2.01  & $2.2e{-2}$  &  - & - &- & $0.350/0.432/0.535$ & -\\
        ROM-RHC $\ubar \alpha_k^{r}$   &2.01 & $2.2e{-2}$ & $1.0e{-7}$ & $1.5e{-6}$ &$1.8e{-7}$&$0.350/0.431/0.535$& $3.4e{-7}$ \\
        ROM-RHC $\ubar \alpha_k$   &2.01 & $2.2e{-2}$ & $1.0e{-7}$ & $1.5e{-6}$ & $1.8e{-7}$ &$ 0.350/0.432/0.535$& $9.7e{-7}$\\
		\midrule
        $(T,\tilde \alpha)=(1.0,0.58)$ & $J^\star_{T_\infty}$ & $|y^\star_{rh}(T_\infty)|_H$ & $e_{J_{T_\infty}}$  & $e_{\bu_{rh}}$ & $e_{y_{rh}}$   & min/avg/max($\ubar\alpha_k$) & $e^{(r)}_\alpha$ \\  
        \midrule
		FOM-RHC & 1.60  & $3.3e{-3}$  &  - & - &- & $0.584/0.630/0.689$ &-\\
        ROM-RHC $\ubar \alpha_k^{r}$   &1.60 & $3.3e{-3}$ & $6.0e{-8}$ & $6.6e{-7}$ & $1.6e{-7}$ &$0.584/0.630/0.689$&$5.4e{-5}$\\
        ROM-RHC $\ubar \alpha_k$   &1.60  & $3.3e{-3}$ & $5.4e{-8}$ & $6.2e{-7}$ &$1.6e{-7}$&$0.584/0.630/0.689$&$7.0e{-7}$\\
        \midrule
        $(T,\tilde \alpha)=(1.2,0.73)$  & $J^\star_{T_\infty}$ & $|y^\star_{rh}(T_\infty)|_H$ & $e_{J_{T_\infty}}$  & $e_{\bu_{rh}}$ & $e_{y_{rh}}$   & min/avg/max($\ubar\alpha_k$) & $e^{(r)}_\alpha$ \\  
        \midrule
		FOM-RHC & 1.46  & $8.6e{-4}$  &  - & - &- & $0.736/0.765/0.799$ &\\
        ROM-RHC $\ubar \alpha_k^{r}$   &1.46  &  $8.6e{-4}$& $7.2e{-8}$ & $6.4e{-7}$ & $4.6e{-7}$ & $0.736/0.764/0.764$&$5.2e{-7}$\\
        ROM-RHC $\ubar \alpha_k$ &1.46  & $8.6e{-4}$ & $7.3e{-8}$ & $6.4e{-7}$ &$4.6e{-7}$& $0.736/0.764/0.799$&$5.1e{-7}$\\
		\bottomrule
	\end{tabular}
    
\end{table}

\begin{table}[ht!] 
    \scriptsize
	\centering 
    \caption{Performance comparison of the RHC schemes. {Basis sizes $r$ are listed in update order.}}
	\label{fig:RHC_perfromacne}
	\begin{tabular}{lccccc}\toprule
		$(T,\tilde \alpha)=(0.8,0.35)$  & $\#$FOM gradient eval. & CPU time [s] & speed-up & \#ROM updates &  $r$  \\ 
        \midrule
		FOM-RHC & 476  & 638  &  - & - & - \\
        ROM-RHC $\ubar \alpha_k^{r}$  &24 & 62 & 10 & 2 &{41/60}\\
        ROM-RHC $\ubar \alpha_k$  &24 & 55 & 11 & 2 &{41/60}\\
		 \midrule
        $(T,\tilde \alpha)=(1.0,0.58)$  & $\#$FOM gradient eval. & CPU time [s] & speed-up & \#ROM updates &  $r$  \\ 
        \midrule
		FOM-RHC & 558  & 907  &  - & - & - \\
        ROM-RHC $\ubar \alpha_k^{r}$  &24 & 75 & 12 & 2 &{43/59}\\
        ROM-RHC $\ubar \alpha_k$  &24 & 67 & 13 & 2 &{43/60}\\
         \midrule
        $(T,\tilde \alpha)=(1.2,0.73)$  & $\#$FOM gradient eval. & CPU time [s] & speed-up & \#ROM updates &  $r$  \\ 
        \midrule
		FOM-RHC & 648  & 1244  &  - & - & - \\
        ROM-RHC $\ubar \alpha_k^{r}$  &31 & 103 & 11 & 2 &{44/60}\\
        ROM-RHC $\ubar \alpha_k$ &31 & 99 & 12 & 2 &{44/60}\\
		\bottomrule
	\end{tabular}
\end{table}
\begin{figure}[htbp]
	\centering
	\begin{subfigure}[t]{0.48\textwidth}
        \centering
	    \includegraphics{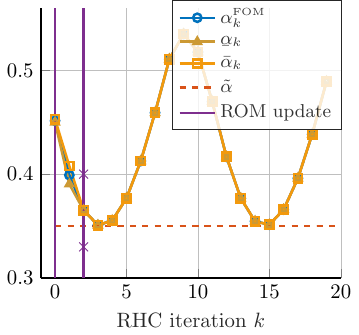}
        \label{fig:Perf_sub1}
	\end{subfigure}
 	\hfill
    \begin{subfigure}[t]{0.48\textwidth}
        \centering%
	    \includegraphics{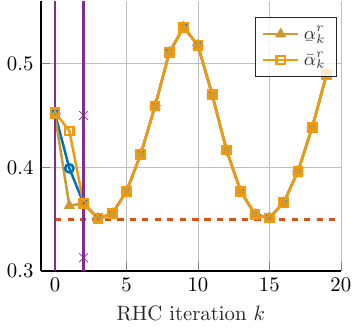}
        \label{fig:Perf_sub2}
	\end{subfigure}
	\caption{Performance index with lower and upper bounds for $(T,\tilde \alpha)=(0.8,0.35)$. {The purple 'x' shows the performance bounds before model refinement.}}
	\label{fig:Perf}
\end{figure}
\begin{figure}[htbp]
	\centering
	\begin{subfigure}[t]{0.48\textwidth}
        \centering
        \includegraphics{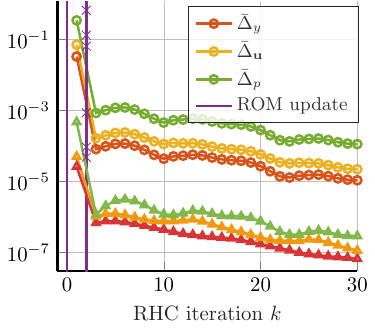}
	\end{subfigure}
 	\hfill
    \begin{subfigure}[t]{0.48\textwidth}
        \centering%
        \includegraphics{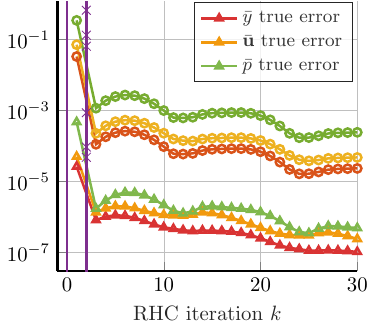}

	\end{subfigure}
	\caption{{State, control, and adjoint state estimates together with their true error for $T=0.8$. Left: ROM-RHC with $ \bar \alpha_k$. Right: ROM-RHC with $ \bar \alpha_k^{r}$. The purple 'x' shows the quantities before model refinement.}}
	\label{fig:mpcest}
\end{figure}
\section{Conclusion}\label{sec:conclusion}
We proved that continuous-time RHC can achieve exponential stability and suboptimality for linear time-varying parabolic equations within a relaxed dynamic programming framework. Using Galerkin reduced-order models and rigorous a posteriori error estimates for the value function, we designed reduced-order controllers that stabilize the full-order system, with convergence guarantees and validated performance in numerical tests involving finite actuator configurations and squared $\ell_1$-regularization. Future work will address extensions to nonlinear parabolic systems, systems with parametric or dynamic uncertainties, {and adaptive choices of the horizon $T$.}
\section*{Acknowledgment} M.K. and S.V. acknowledge funding by the Federal Ministry of Education and Research (grant no. 05M22VSA).
%
\bibliographystyle{siam}
\bibliography{biblio}

\begin{thebibliography}{10}

\bibitem{ali2020reduced}
{\sc A.~A. Ali and M.~Hinze}, {\em Reduced basis methods -- an application to variational discretization of parametrized elliptic optimal control problems}, SIAM Journal on Scientific Computing, 42 (2020), pp.~A271--A291.

\bibitem{AGH22}
{\sc A.~Alla, C.~Gräßle, and M.~Hinze}, {\em Time adaptivity in model predictive control}, J. Sci. Comput., 90 (2022), pp.~Paper No. 12, 24.

\bibitem{AV15}
{\sc A.~Alla and S.~Volkwein}, {\em Asymptotic stability of {POD} based model predictive control for a semilinear parabolic {PDE}}, Advances in Computational Mathematics, 41 (2015), p.~1073–1102.

\bibitem{APKCSZHP23}
{\sc J.~I. Alora, L.~A. Pabon, J.~Köhler, M.~Cenedese, E.~Schmerling, M.~N. Zeilinger, G.~Haller, and M.~Pavone}, {\em Robust nonlinear reduced-order model predictive control}, in 2023 62nd IEEE Conference on Decision and Control (CDC), 2023, pp.~4798--4805.

\bibitem{azmi2025forward}
{\sc B.~Azmi and M.~Bernreuther}, {\em On the forward--backward method with nonmonotone linesearch for infinite-dimensional nonsmooth nonconvex problems}, Computational Optimization and Applications, 91 (2025), pp.~1263--1308.

\bibitem{azmi2019hybrid}
{\sc B.~Azmi and K.~Kunisch}, {\em A hybrid finite-dimensional {RHC} for stabilization of time-varying parabolic equations}, SIAM Journal on Control and Optimization, 57 (2019), pp.~3496--3526.

\bibitem{ARV25}
{\sc B.~Azmi, J.~Rohleff, and S.~Volkwein}, {\em Finite-dimensional receding horizon control of linear time-varying parabolic {PDE}s: stability analysis and model-order reduction}, Springer Nature Switzerland, Cham, 2025, pp.~55--81.

\bibitem{ran}
{\sc R.~Becker and R.~Rannacher}, {\em An optimal control approach to a posteriori error estimation in finite element methods}, Acta numerica, 10 (2001), pp.~1--102.

\bibitem{BF20}
{\sc T.~Breiten and L.~Pfeiffer}, {\em On the turnpike property and the receding-horizon method for linear-quadratic optimal control problems}, SIAM J. Control Optim., 58 (2020), pp.~1077--1102.

\bibitem{cohen2011analytic}
{\sc A.~Cohen, R.~Devore, and C.~Schwab}, {\em Analytic regularity and polynomial approximation of parametric and stochastic elliptic pde's}, Analysis and Applications, 9 (2011), pp.~11--47.

\bibitem{dietze2023reduced}
{\sc S.~Dietze and M.~A. Grepl}, {\em Reduced order model predictive control for parametrized parabolic partial differential equations}, Applied Mathematics and Computation, 453 (2023), p.~128044.

\bibitem{GU14}
{\sc J.~Ghiglieri and S.~Ulbrich}, {\em Optimal flow control based on {POD} and {MPC} and an application to the cancellation of {T}ollmien–{S}chlichting waves}, Optimization Methods and Software, 29 (2014), pp.~1042--1074.

\bibitem{G09}
{\sc L.~Gr\"une}, {\em Analysis and design of unconstrained nonlinear {MPC} schemes for finite and infinite dimensional systems}, SIAM J. Control Optim., 48 (2009), pp.~1206--1228.

\bibitem{grune2017nonlinear}
{\sc L.~Gr{\"u}ne and J.~Pannek}, {\em Nonlinear Model Predictive Control: Theory and Algorithms}, Communications and Control Engineering, Springer Cham, second~ed., 2017.

\bibitem{GR08}
{\sc L.~Gr\"une and A.~Rantzer}, {\em On the infinite horizon performance of receding horizon controllers}, IEEE Trans. Automat. Control, 53 (2008), pp.~2100--2111.

\bibitem{GSS22}
{\sc L.~Gr\"une, M.~Schaller, and A.~Schiela}, {\em Efficient model predictive control for parabolic {PDE}s with goal oriented error estimation}, SIAM J. Sci. Comput., 44 (2022), pp.~A471--A500.

\bibitem{GubV17}
{\sc M.~Gubisch and S.~Volkwein}, {\em Proper orthogonal decomposition for linear-quadratic optimal control}, in Model Reduction and Approximation: Theory and Algorithms, Society for Industrial and Applied Mathematics, 2017, pp.~3--63.

\bibitem{hesthaven2016certified}
{\sc J.~S. Hesthaven, G.~Rozza, B.~Stamm, et~al.}, {\em Certified Reduced Basis Methods for Parametrized Partial Differential Equations}, vol.~590, Springer, 2016.

\bibitem{HKMU23}
{\sc M.~Hinze, N.~Kutz, O.~Mula, and K.~Urban}, {\em Model Order Reduction and Applications. Cetraro, Italy 2021}, Springer, Cham, 2023.

\bibitem{hinze2008optimization}
{\sc M.~Hinze, R.~Pinnau, M.~Ulbrich, and S.~Ulbrich}, {\em Optimization with PDE Constraints}, vol.~23, Springer Science \& Business Media, 2008.

\bibitem{IK02}
{\sc K.~Ito and K.~Kunisch}, {\em Receding horizon optimal control for infinite dimensional systems}, ESAIM Control Optim. Calc. Var., 8 (2002), pp.~741--760.
\newblock A tribute to J. L. Lions.

\bibitem{JH05}
{\sc A.~Jadbabaie and J.~Hauser}, {\em On the stability of receding horizon control with a general terminal cost}, IEEE Trans. Automat. Control, 50 (2005), pp.~674--678.

\bibitem{karcher2018certified}
{\sc M.~K{\"a}rcher, Z.~Tokoutsi, M.~A. Grepl, and K.~Veroy}, {\em Certified reduced basis methods for parametrized elliptic optimal control problems with distributed controls}, Journal of Scientific Computing, 75 (2018), pp.~276--307.

\bibitem{kartmann2024certifiedmodelpredictivecontrol}
{\sc M.~Kartmann, M.~Manucci, B.~Unger, and S.~Volkwein}, {\em Certified model predictive control for switched evolution equations using model order reduction}, arXiv,  (2024).
\newblock To appear in Communications on Applied Mathematics and Computation, 2026.

\bibitem{KV02}
{\sc K.~Kunisch and S.~Volkwein}, {\em Galerkin proper orthogonal decomposition methods for a general equation in fluid dynamics}, SIAM J. Numer. Anal., 40 (2002), pp.~492--515.

\bibitem{LioM72}
{\sc J.~L. Lions and E.~Magenes}, {\em Non-Homogeneous Boundary Value Problems and Applications}, Grundlehren der mathematischen Wissenschaften, Springer Berlin, Heidelberg, 1972.

\bibitem{loehning2014model}
{\sc M.~Loehning, M.~Reble, J.~Hasenauer, S.~Yu, and F.~Allg{\"o}wer}, {\em Model predictive control using reduced order models: guaranteed stability for constrained linear systems}, J. Process Control., 24 (2014), pp.~1647--1659.

\bibitem{LMFP22}
{\sc J.~Lorenzetti, A.~McClellan, C.~Farhat, and M.~Pavone}, {\em Linear reduced-order model predictive control}, IEEE Transactions on Automatic Control, 67 (2022), pp.~5980--5995.

\bibitem{MU22}
{\sc P.~Manns and S.~Ulbrich}, {\em A simplified {N}ewton method to generate snapshots for {POD} models of semilinear optimal control problems}, SIAM J. Numer. Anal., 60 (2022), pp.~2807--2833.

\bibitem{RMD19}
{\sc J.~B. Rawlings, D.~Q. Mayne, and M.~M. Diehl}, {\em Model Predictive Control: Theory, Computation, and Design}, Nob Hill Publishing, 2~ed., 2019.
\newblock 1st printing.

\bibitem{RA12}
{\sc M.~Reble and F.~Allg\"ower}, {\em Unconstrained model predictive control and suboptimality estimates for nonlinear continuous-time systems}, Automatica J. IFAC, 48 (2012), pp.~1812--1817.

\bibitem{vexler2008adaptive}
{\sc B.~Vexler and W.~Wollner}, {\em Adaptive finite elements for elliptic optimization problems with control constraints}, SIAM Journal on Control and Optimization, 47 (2008), pp.~509--534.

\end{thebibliography}
%
\appendix
%
\section{Error estimator equivalence and asymptotic convergence}\label{sec:APP_add_estimates}
%
\begin{lemma}[Error estimator equivalence in $\mY_T(\tint)$]\label{lemma:app_error_equivalence}
    Let \cref{ass:pde} hold and let $\tint\in\R_{\geq0}$, $T\in \R_{>0}$.
    \begin{enumerate}
        \item In the situation of \cref{Lemma:ResBasStateError}, let $y=y(\bu,\yint)\in \mY_T(\tint),\  y^r=y^r(\bu,\yint)\in \mY^r_T(\tint)$ be the solution of \eqref{eq:LTV} and \eqref{eq:reducedLTV}, respectively, for $\tildeyint=\yint\in H$, $\bu\in \mU_T(\tint)$. For $e_y = y-y^r$ it holds
     \begin{equation}\label{eq:state_errorest_equiv}
     c \Delta^2_{y}(0,0, y^r,\bu)\leq \norm{e_y(\tint+T)}{H}^2+\norm{e_y}{\mY_T(\tint)}^2\leq C\Delta^2_y(0,0,y^r, \bu),
    \end{equation}
    for constants $C,c>0$ independent of $(y,y^r,\bu, \yint,r)$.
    \item In the situation of \cref{lem:ResAdjest}, let $p=p(\tilde y)\in  \mY_T(\tint)$ and $p^r=p^r(\tilde y)\in \mY^r_T(\tint)$ be the solution of the FOM \eqref{eq:FOM_opsys} and ROM \eqref{eq:ROM_opsys} adjoint equation, respectively, for $\tilde y\in L^2(\tint,\tint+T;H)$. For $e_p = p-p^r$ it holds
    \begin{equation}\label{eq:adstate_errorest_equiv}
    \tilde c\Delta^2_{p}(0, p^r, \tilde y) \leq \norm{e_p(\tint)}{H}^2+\norm{e_p}{\mY_T(\tint)}^2\leq \tilde C\Delta^2_{p}(0, p^r, \tilde y),
    \end{equation}
    for constants $\tilde C,\tilde c>0$ independent of $(p,p^r,\tilde y,r)$.
    \end{enumerate}
\end{lemma}
\begin{proof}
    We only prove \eqref{eq:state_errorest_equiv}, since \eqref{eq:adstate_errorest_equiv} follows analogously. Let $J_V:V\to V'$ denote the Riesz isomorphism. For the upper bound, we test the error equation \eqref{eq:error_eq} with $J_V^{-1}\dot e_y(t)$ for almost all $t\in (\tint,\tint+T)$ to obtain
    %
    \begin{equation}\label{eq:lemm:app_error_equivalence_1}
        \norm{\dot e_y}{L^2(V')}\leq \norm{\Res_y}{L^2(V')}+\norm{A}{}\norm{e_y}{L^2(V)}.
    \end{equation}
    for $\norm{A}{L^\infty}\coloneqq \norm{A}{L^\infty(\tint,\tint+T;\calL(V,V'))}$.
    From \eqref{eq:state_apost_CHL2V} for $\Delta_\bu=\Delta_{\yint}=0$, it follows
    \begin{align*}
        &\norm{e_y(\tint+T)}{H}^2+ \norm{e_y}{\mY_T(\tint)}^2 \stackrel{\text{\eqref{eq:state_apost_CHL2V}}}{\leq} C_{1,y}\norm{e_y(\tint)}{H}^2+C_{2,y}\norm{\Res_y}{L^2(V')}^2+\norm{\dot e_y}{L^2(V')}^2\\
        & \qquad\stackrel{\eqref{eq:lemm:app_error_equivalence_1}}{\leq}  C_{1,y}\norm{e_y(\tint)}{H}^2+C_{2,y}\norm{\Res_y}{L^2(V')}^2+ 2\norm{\Res_y}{L^2(V')}^2+2\norm{A}{L^\infty}^2\norm{e_y}{L^2(V)}^2\\
        &\qquad\stackrel{\text{\eqref{eq:state_apost_CHL2V}}}{\leq} \big(2\norm{A}{L^\infty}^2C_{1,y}+C_{1,y}\big)\norm{e_y(\tint)}{H}^2+ \big(2\norm{A}{L^\infty}^2C_{2,y}+2+C_{2,y})\norm{\Res_y}{L^2(V')}^2,
    \end{align*}
    which implies the upper bound in \eqref{eq:state_errorest_equiv}. For the lower bound, we test the error equation \eqref{eq:error_eq} with $J_V^{-1}\Res_y(t)$ for almost all $t\in (\tint,\tint+T)$ to obtain
    \begin{equation}\label{eq:lemm:app_error_equivalence_2}
        \norm{\Res_y}{L^2(V')}\leq \norm{A}{L^\infty} \norm{e_y}{L^2(V)}+\norm{\dot e_y}{L^2(V')}.
    \end{equation}
    On the other hand, we have
    \begin{align}
       \label{eq:lemm:app_error_equivalence_3}
        \begin{aligned}
            \norm{e_y(\tint)}{H}^2&=\norm{e_y(\tint+T)}{H}^2-2\int_{\tint}^{\tint+T}\langle \dot e_y, e_y\rangle_{V',V}\ \dt\\
            &\leq \norm{e_y(\tint+T)}{H}^2 + 2\norm{\dot e_y}{L^2(V')}^2+2\norm{ e_y}{L^2(V)}^2.
        \end{aligned}
    \end{align}
    Combining \eqref{eq:lemm:app_error_equivalence_2} and \eqref{eq:lemm:app_error_equivalence_3}, implies the lower bound in \eqref{eq:state_errorest_equiv} from
    \begin{equation*}
         \norm{e_y(\tint)}{H}^2+\norm{\Res_y}{L^2(V')}^2\leq \norm{e_y(\tint+T)}{H}^2 + 3\norm{\dot e_y}{L^2(V')}^2 + \big(2+\norm{A}{L^\infty}\big)\norm{e_y}{L^2(V)}. 
    \end{equation*}
\end{proof}
Next, we show that the errors \( e_y, e_p \) converge to zero in \( \mY_T(\tint) \) for an orthonormal Galerkin projection, as \( r \to \infty \), under the regularity assumption \( y, p \in \mY_T(\tint) \). This convergence, as well as convergence rates, were established in~\cite[Theorem~3.11]{GubV17}, for instance, under the stronger regularity assumption \( y, p \in H^1(\tint,\tint+T;V) \). In our setting, convergence in $\mY_T(\tint)$ is essential in order to ensure convergence of the residuals and, consequently, the decay of residual-based a posteriori error estimates.
\begin{theorem}[$\mY_T(\tint)$-convergence of the state and adjont state]\label{theo:APP_convergence}
In the situation of \ \cref{lemma:app_error_equivalence}, 
consider an orthonormal basis $(v_n)_{n\in \N}$ of $V$. We set $V_r\coloneqq \text{span}\{v_1,\ldots,v_r\}$. Then we have $\norm{ e_y}{\mY_T(\tint)}\to 0$, $\norm{ e_p}{\mY_T(\tint)}\to 0$ as $r\to \infty$. 
%
\end{theorem}
\begin{proof}
    We show $\norm{e_y}{\mY_T(\tint)}\to 0$, as the claim for $e_p$ follows by similar arguments. 
    Consider the orthogonal projection operator $\Pi^V_{V_r}:V\to V_r$ defined below \eqref{eqn:PODminimization}. We decompose the error according to $ y^r-y = y^r-\Pi^V_{V_r} y+\Pi^V_{V_r} y-y\eqqcolon e_2 + e_1$. Consider $e_1$ first. It holds $\ddt(\Pi^V_{V_r} y)= \dot y \circ \Pi^V_{V_r} $ (since $\Pi^V_{V_r}$ is self-adjoint as an orthogonal projection) and therefore
    \begin{align*}
        \norm{e_1}{\mY(\tint)}^2=&\norm{\Pi^V_{V_r} y-y}{L^2(V)}^2+\norm{\dot y \circ \Pi^V_{V_r} -\dot y}{L^2(V')}^2\\
        \leq &  \norm{\Pi^V_{V_r} y-y}{L^2(V)}^2+\norm{\dot y}{L^2(V')}^2\sup_{v\in L^2(V), \norm{v}{L^2(V)}=1}\norm{\Pi^V_{V_r}v-v}{L^2(V)}^2\to 0 \ (r\to \infty).
    \end{align*}
    Now we turn to $e_2$. It holds for $v\in V_r$ and almost all $t\in (\tint,\tint+T)$
    \begin{align}\label{eq:theo:galerkin_conv_help0}
        \langle\dot e_2(t),v\rangle_{V',V}+\langle A(t)e_2(t), v\rangle_{V',V}=& \langle-\dot e_1(t),v\rangle_{V',V}+\langle -A(t)e_1(t), v\rangle_{V',V}\\ \nonumber
        \leq & \norm{\dot e_1(t)}{V'}\norm{v}{V}+\norm{A}{L^\infty}\norm{e_1(t)}{V}\norm{v}{V}
    \end{align}
    By choosing $v=e_2(t)$ and using the weak coercivity of the operator $A$ as stated in \eqref{eq:weak_coercivity}, along with Young's inequality, we arrive at
    \begin{align}\label{eq:theo:galerkin_conv_help1}
         \tfrac{1}{2}\norm{\dot e_2(t)}{H}^2+\tfrac{\eta_V}{2}\norm{e_2(t)}{V}^2
        \leq & \tfrac{1}{\eta_V}\norm{\dot e_1(t)}{V'}^2+\tfrac{\norm{A}{L^\infty}}{\eta_V} \norm{e_1(t)}{V}^2+\eta_H\norm{e_2(t)}{H}^2
    \end{align}
    By applying Gronwall's Lemma and integration over $(\tint,\tint+T)$ it follows that
    %
    \begin{align}\label{eq:theo:galerkin_conv_help2}
        \norm{e_2}{L^2(V)}^2\leq c\norm{\dot e_1}{\mY_T(\tint)}^2+c\norm{e_2(\tint)}{H}^2.
    \end{align}
    for a generic constant $c=c(T)$.
    For $e_2(\tint)$, we obtain, using $\mY_T(\tint)\hookrightarrow C([\tint,\tint+T];H)$, that
    \begin{align*}
      \norm{e_2(\tint)}{H}^2=&\norm{\big(y^r-y+y-\Pi^V_{V_r} y\big)(\tint)}{H}^2\leq  \norm{y^r(\tint)-\yint}{H}^2+ c \norm{y-\Pi^V_{V_r} y}{\mY_T(\tint)}^2\\
      = & \norm{\Pi^H_{V_r} \yint-\yint}{H}^2+c \norm{e_1}{\mY_T(\tint)}^2 \to 0 \ (r\to \infty).
    \end{align*}
    Here we used the fact, that $\Pi^H_{V_r}\yint-\yint \to 0 \ (r\to\infty)$, since $V\subset H$ dense.
    Hence, \eqref{eq:theo:galerkin_conv_help2} implies $e_2\to 0$ in $L^2(V)$ as $r\to\infty$. Next, we show $\norm{\dot e_2}{L^2(V')}=\norm{\dot e_2}{L^2(V_r')}$. We have $y^r\in H^1(V_r)$ and hence $\dot y^r(t)\in V_r\subset V\subset V'\subset V_r'$ for almost all $t\in (\tint,\tint+T)$. Therefore, we conclude the result by applying Riesz's representation theorem
    \begin{equation}\label{eq:hh0}
        \norm{\dot y^r(t)}{V}= \norm{\dot y^r(t)}{V_r}=\norm{\dot y^r(t)}{V_r'}\leq \norm{\dot y^r(t)}{V'}=\norm{\dot y^r(t)}{V}.
    \end{equation}
    Further, we have using $\norm{\Pi^V_{V_r}v }{V}\leq \norm{v}{V}$ for all $v\in V$ and $\Pi_{V_r}^V\circ \Pi_{V_r}^V= \Pi_{V_r}^V $
    \begin{align}\nonumber
       \norm{\dot y(t)\circ  \Pi_{V_r}^V}{V_r'} \leq \norm{\dot y(t)\circ  \Pi_{V_r}^V}{V'}=& \sup_{ v\in V\setminus\{0\}}\tfrac{| \langle \dot y(t), \Pi_{V_r}^Vv\rangle_{V',V}|}{\norm{v}{V}}\leq\sup_{ v\in V\setminus\{0\}}\tfrac{| \langle \dot y(t), \Pi_{V_r}^Vv\rangle_{V',V}|}{\norm{ \Pi_{V_r}^Vv}{V}}\\
        =&\sup_{ v\in V_r\setminus\{0\}}\tfrac{| \langle \dot y(t)\circ  \Pi_{V_r}^V,v\rangle_{V',V}|}{\norm{v}{V}}= \norm{\dot y(t)\circ \Pi_{V_r}^V}{V_r'}.\label{eq:hh1}
    \end{align}
    Hence,
    \begin{align}\nonumber
        \norm{\dot e_2}{L^2(V')}^2 =& \norm{\dot y^r}{L^2(V')}^2-2\langle \dot y^r,  \dot y\circ  \Pi_{V_r}^V \rangle_{L^2(V')}+\norm{\dot y\circ  \Pi_{V_r}^V}{L^2(V')}^2\\
        \stackrel{\eqref{eq:hh0},\eqref{eq:hh1}}{=}&\norm{\dot y^r}{L^2(V_r')}^2-2\langle \dot y^r,  \dot y\circ  \Pi_{V_r}^V \rangle_{L^2(V_r')}+\norm{\dot y\circ \Pi_{V_r}^V}{L^2(V_r')}^2=\norm{\dot e_2}{L^2(V_r')}\label{eq:hh2}
    \end{align}
    Now we can estimate using \eqref{eq:theo:galerkin_conv_help0} as
    \begin{align*}
         \norm{\dot e_2}{L^2(V')}\stackrel{\eqref{eq:hh2}}{=} \norm{\dot e_2}{L^2(V_r')}
         &\stackrel{\eqref{eq:theo:galerkin_conv_help0}}{=}\sup_{v\in V_r, \norm{v}{L^2(V)}=1}|\langle \dot e_1+Ae_1-Ae_2,v \rangle_{L^2(V'),L^2(V)}|\\
         & \leq c \norm{\dot e_1}{L^2(V')}+ c\norm{e_1}{L^2(V)}+c\norm{e_2}{L^2(V)}\to 0 \ (r\to \infty).
    \end{align*}
\end{proof}
\begin{proof}[Proof of \cref{theo:APP_convergence}]
    First, we show $\bar u^r \to \bar u$. Choosing $\bu= \bar \bu^r$ in \eqref{eq:FOML3}, $\bu= \bar \bu$ in \eqref{eq:ROML3} and adding the two equations results in
    \begin{equation}\nonumber 
        \lambda \norm{ \bar \bu^r-\bar \bu}{\mU_T(\tint)}^2\leq \langle B'(\bar p^r-\bar p),\bar \bu-\bar \bu^r\rangle_{\mU_T(\tint)}.
    \end{equation}
    Adding $\pm p^r(\bar y)$, results in
     \begin{align}\label{eq:app_con_os}
        \lambda \norm{ \bar \bu^r-\bar \bu}{\mU_T(\tint)}^2\leq \langle B'(\bar p^r-p^r(\bar y)),\bar \bu-\bar \bu^r\rangle_{\mU_T(\tint)}+ \langle B'(p^r(\bar y)-\bar p),\bar \bu-\bar \bu^r\rangle_{\mU_T(\tint)}
    \end{align}
    For the first term, we obtain, using the ROM state equation, partial integration, the ROM adjoint equation, adding $\pm y^r(\bar \bu)$, and Young's inequality
    \begin{align*}
        \langle B'(\bar p^r-p^r(\bar y)),\bar \bu-\bar \bu^r\rangle_{\mU_T(\tint)} = & \langle \bar y^r-\bar y, y^r(\bar \bu)-\bar y^r\rangle_{L^2(H)}\\
        \leq &-\tfrac{1}{2}\norm{\bar y^r-y^r(\bar \bu)}{L^2(H)}^2 + \tfrac{1}{2}\norm{\bar y-y^r(\bar \bu)}{L^2(H)}^2.
    \end{align*}
   For the second term in \eqref{eq:app_con_os}, we have by Young's inequality
    \begin{equation}\nonumber
        \langle B'(p^r(\bar y)-\bar p),\bar \bu-\bar \bu^r\rangle_{\mU_T(\tint)}\leq \tfrac{\norm{B}{L^\infty}^2}{2\lambda}\norm{p^r(\bar y)-\bar p}{\mY_T(\tint)}^2+\tfrac{1}{2\lambda}\norm{\bar \bu-\bar \bu^r}{\mU_T(\tint)}^2
    \end{equation}
    Inserting this into into \eqref{eq:app_con_os}, yields
    \begin{equation}\nonumber 
        \lambda \norm{ \bar \bu^r-\bar \bu}{\mU_T(\tint)}^2+\norm{\bar y^r-y^r(\bar \bu)}{L^2(H)}^2\leq \norm{\bar y-y^r(\bar \bu)}{L^2(H)}^2+\tfrac{\norm{B}{L^\infty}^2}{2\lambda}\norm{p^r(\bar y)-\bar p}{\mY_T(\tint)}^2\to 0
    \end{equation}
     for $r\to \infty$, due to \cref{theo:APP_convergence} for $\bu=\bar \bu$ and $\tilde y = \bar y$, respectively. Now $\bar y^r \to \bar y$ in $\mY_T(\tint)$, follows from the decomposition $\bar y^r-\bar y=\bar y^r-y^r(\bar \bu)+y^r(\bar \bu)-\bar y\eqqcolon e_1+e_2$. For $e_1$, it holds by standard a priori estimates $\norm{e_1}{\mY_T(\tint)}\leq C\norm{\bar \bu-\bar \bu^r}{\mU_T(\tint)}\to 0$ for $C>0$ independent of $r$, and $e_2\to 0$ in $\mY_T(\tint)$ by \cref{theo:APP_convergence} for $\bu=\bar \bu$. With similar arguments, $\bar p^r- \bar p=\bar p^r-p^r(\bar y )+p^r(\bar y)-\bar p$ implies $\bar p^r\to \bar p$ in $\mY_{T}(\tint)$.
\end{proof}
\end{document}